\documentclass[reqno,12pt]{amsart}
\usepackage{amssymb}
\usepackage{amsfonts}
\usepackage{amsmath}
\usepackage{graphicx}
\usepackage{amscd,amsthm}

\newtheorem{theorem}{Theorem}
\theoremstyle{plain}
\newtheorem{acknowledgement}{Acknowledgement}

\newtheorem{corollary}{Corollary}

\newtheorem{lemma}{Lemma}

\newtheorem{remark}{Remark}

\numberwithin{equation}{section}

\begin{document}
\author{}
\title{}
\maketitle

\begin{center}
\thispagestyle{empty} \pagestyle{myheadings} 
\markboth{\bf Yilmaz Simsek
}{\bf Computaion negative-order Euler numbers }

\textbf{{\Large New families of special numbers for computing \textbf{%
negative order} Euler numbers}}

\bigskip

\textbf{Yilmaz Simsek}\\[0pt]

Department of Mathematics, Faculty of Science University of Akdeniz TR-07058
Antalya, Turkey,

\bigskip

ysimsek@akdeniz.edu.tr\\[0pt]

\bigskip

\textbf{{\large {Abstract}}}\medskip
\end{center}

\begin{quotation}
The main purpose of this paper is to construct new families of special
numbers with their generating functions. These numbers are related to the
many well-known numbers, which are the Bernoulli numbers, the Fibonacci
numbers, the Lucas numbers, the Stirling numbers of the second kind and the
central factorial numbers. Our other inspiration of this paper is related to
the Golombek's problem \cite{golombek} \textquotedblleft Aufgabe 1088, El.
Math. 49 (1994) 126-127\textquotedblright . Our first numbers are not only
related to the Golombek's problem, but also computation of the negative
order Euler numbers. We compute a few values of the numbers which are given
by some tables. We give some applications in Probability and Statistics.
That is, special values of mathematical expectation in the binomial
distribution and the Bernstein polynomials give us the value of our numbers.
Taking derivative of our generating functions, we give partial differential
equations and also functional equations. By using these equations, we derive
recurrence relations and some formulas of our numbers. Moreover, we give two
algorithms for computation our numbers. We also give some combinatorial
applications, further remarks on our numbers and their generating functions.
\end{quotation}

\noindent \textbf{2010 Mathematics Subject Classification.} 12D10, 11B68,
11S40, 11S80, 26C05, 26C10, 30B40, 30C15.

\bigskip

\noindent \textbf{Key Words.} Bernoulli numbers; Euler numbers; Central
factorial numbers; Array polynomials; Stirling numbers; Generating
functions; Functional equations; Binomial coefficients; Binomial sum;
Combinatorial sum.

\section{Introduction}

In this section, we consider the following question:

What could be a more basic tools to compute the negative order of the first
and the second kind Euler numbers? One of the motivations is associated with
this question and its answer. Especially, in the work of Golombek \cite%
{golombek}, which is entitled \textbf{Aufgabe 1088}, we see the following
novel combinatorial sum:%
\begin{equation}
\sum_{j=0}^{k}\left( 
\begin{array}{c}
k \\ 
j%
\end{array}%
\right) j^{n}=\frac{d^{n}}{dt^{n}}\left( e^{t}+1\right) ^{k}\left\vert
_{t=0}\right. ,  \label{Gl}
\end{equation}%
where $n=1,2,\ldots $. Golombek \cite{golombek} also mentioned that this sum
is related to the following sequence%
\begin{equation*}
n2^{n-1},n(n+1)2^{n-2},\ldots
\end{equation*}

The other motivation is to introduce new families of special numbers, which
are not only used in counting techniques and problems, but also computing
negative order of the first and the second kind Euler numbers and other
combinatorial sums. Here, our technique is related to the generating
functions. In the historical development of Mathematics, we can observe that
the generating functions play a very important role in Pure and Applied
Mathematics. These function are powerful tools to solve counting problems
and investigate properties of the special numbers and polynomials. In
addition, the generating functions are also used in Computer Programming, in
Physics, and in other areas. Briefly, in Physics, generating functions,
which arise in Hamiltonian mechanics, are quite different from generating
functions in mathematics. The generating functions are functions whose
partial derivatives generate the differential equations that determine a
system's dynamics. These functions are also related to the partition
function of statistical mechanics (\textit{cf}. \cite{Charamb}, \cite{Knuth}%
, \cite{M. Spiegel}, \cite{Wm}, \cite{Wp}). As for mathematics, a generating
function is a formal power series in one indeterminate whose coefficients
encode information about a sequence of numbers and that is indexed by the
natural numbers (\textit{cf}. \cite{Charamb}, \cite{Comtet}, \cite{Doubilet}%
, \cite{Grademir}, \cite{Knuth}, \cite{M. Spiegel}, \cite{Rainville}, \cite%
{SrivastavaBook}, \cite{Wm}). As far as we know that the generating function
is firstly discovered by Abraham de Moivre (26 May 1667 -27 November 1754,
French mathematician) (\textit{cf}. \cite{Knuth}, \cite{Wm}). In order to
solve the general linear recurrence problem, Moivre constructed the concept
of the generating functions in 1730. In work of Doubilet et al. \cite%
{Doubilet}, we also see that Laplace (23 March 1749-5 March 1827, French
mathematician, physicist and statistician), discovered the remarkable
correspondence between set theoretic operations and operations on formal
power series. Their method gives us great success to solve a variety of
combinatorial problems. They developed new kinds of algebras of generating
functions better suited to combinatorial and probabilistic problems. Their
method is depended on group algebra (or semigroup algebra) (see for detail 
\cite{Doubilet}). It is well-known that there are many different ways\ or
approaches to generate a sequence of numbers and polynomials from the series
or the generating functions. The purpose of this paper is to construct the
generating functions for new families of numbers involving Golombek's
identity in (\ref{Gl}), the Stirling numbers, the central factorial numbers,
the Euler numbers of negative order, the rook numbers and combinatorial
sums. Our method and approach provides a way of constructing new special
families of numbers and combinatorial sums. We show how several of these
numbers and these combinatorial sums relate to each other.

\textit{We summarize our paper results as follows}:

In Section 2, we briefly review some special numbers and polynomials, which
are the Bernoulli numbers, the Euler numbers, the Stirling numbers, the
central factorial numbers and the array polynomials.

In Section 3, we give a generating function. By using this function, we
define a family of new numbers $y_{1}(n,k;\lambda )$. We investigate many
properties including a recurrence relation of these numbers by using their
generating functions. We compute a few values of the numbers $%
y_{1}(n,k;\lambda )$, which are given by the tables. We give some remarks
and comments related to the Golombek's identity and the numbers $%
y_{1}(n,k;1) $. Finally, we give a conjecture with two open questions.

In Section 4, we give a generating function for a new family of the other
numbers $y_{2}(n,k;\lambda )$. By using this function, we investigate many
properties with a recurrence relation of these numbers. We compute a few
values of the numbers $y_{2}(n,k;\lambda )$, which are given by the tables.
We give relations between these numbers, the Fibonacci numbers, the Lucas
numbers, and the $\lambda $-Stirling numbers of the second kind. We also
give some combinatorial sums.

In Section 5, we define $\lambda $-central factorial numbers $C(n,k;\lambda
) $. By using their generating functions, we derive some identities and
relations for these numbers and the others.

In Section 6, we give some applications related to the special values of
mathematical expectation for the binomial distribution, the Bernstein
polynomials and the Bernoulli polynomials.

In Section 7, by using the numbers $y_{1}(n,k;\lambda )$, we compute the
Euler numbers of negative order. In addition, we compute a few values of
these numbers, which are given by the tables.

In Section 8, We give two algorithms for our computations.

In Section 9, we give some combinatorial applications, including a rook
numbers and polynomials. We also give combinatorial interpretation for the
numbers $y_{1}(n,k)$. Finally in the last section, we give further remarks
with conclusion.

\textbf{Notations}: Throughout this paper, we use the following standard
notations:

\begin{equation*}
\mathbb{N}=\{1,2,3,\ldots \}\text{, }\mathbb{N}_{0}=\{0,1,2,3,\ldots \}=%
\mathbb{N}\cup \{0\}
\end{equation*}%
Here, $\mathbb{Z}$ denotes the set of integers, $\mathbb{R}$ denotes the set
of real numbers and $\mathbb{C}$ denotes the set of complex numbers. The
principal value $\ln z$ is the logarithm whose imaginary part lies in the
interval $(-\pi ,\pi ]$. Moreover we also use the following notational
conventions:%
\begin{equation*}
0^{n}=\left\{ 
\begin{array}{cc}
1, & (n=0) \\ 
0, & (n\in \mathbb{N)}%
\end{array}%
\right.
\end{equation*}%
and%
\begin{equation*}
\left( 
\begin{array}{c}
\lambda \\ 
0%
\end{array}%
\right) =1\text{ and }\left( 
\begin{array}{c}
\lambda \\ 
v%
\end{array}%
\right) =\frac{\lambda (\lambda -1)\cdots (\lambda -v+1)}{v!}=\frac{\left(
\lambda \right) _{v}}{v!}\text{ }(n\in \mathbb{N}\text{, }\lambda \in 
\mathbb{C)}
\end{equation*}%
(\textit{cf}. \cite{Bayad}, \cite{Comtet}, \cite{SrivastavaChoi2012}). For
combinatorial example, we will use the notations of Bona \cite{Bona}, that
is the set $\{1,2,\ldots ,n\}$ is an $n$-element set, that is, $n$ distinct
objects. Therefore, Bona introduced the shorter notation $[n]$ for this set.
The number $n(n-1)(n-2)\cdots (n-k+1)$ of all $k$-element lists from $[n]$
without repetition occurs so often in combinatorics that there is a symbol
for it, namely%
\begin{equation*}
(n)_{k}=n(n-1)(n-2)\cdots (n-k+1)
\end{equation*}%
(\textit{cf}. \cite[PP. 11-13.]{Bona}).

\section{Background}

In this section, we give a brief introduction about the Bernoulli numbers,
the Euler numbers, the ($\lambda $-)Stirling numbers and the array
polynomials. Because we use these numbers and polynomials in the next
sections.

In \cite{Alayont}-\cite{Qi}, we see that there are many known properties and
relations involving various kind of the special numbers and polynomials such
as the Bernoulli polynomials and numbers, the Euler polynomials and numbers,
the Stirling numbers and also the rook polynomials and numbers by making use
of some standard techniques based upon generating functions and other known
techniques.

The Bernoulli polynomials are defined by means of the following generating
function:%
\begin{equation*}
\frac{t}{e^{t}-1}e^{tx}=\sum_{n=0}^{\infty }B_{n}(x)\frac{t^{n}}{n!},
\end{equation*}%
where $\left\vert t\right\vert <2\pi $ (\textit{cf}. \cite{Grademir}-\cite%
{Qi}; see also the references cited in each of these earlier works).

One can observe that%
\begin{equation*}
B_{n}=B_{n}(0),
\end{equation*}%
which denotes the Bernoulli numbers\textit{\ (cf. \cite{Grademir}-\cite{Qi};
see also the references cited in each of these earlier works).}

The sum of the powers of integers is related to the Bernoulli numbers and
polynomials:%
\begin{equation}
\sum_{k=0}^{n}k^{r}=\frac{1}{r+1}\left( B_{r+1}(n+1)-B_{r+1}(0)\right) ,
\label{BB11}
\end{equation}%
(\textit{cf}. \cite{Grademir}, \cite{SrivastavaBook}, \cite%
{SrivastavaChoi2012}).

The first kind Apostol-Euler polynomials of order $k$ are defined by means
of the following generating function:%
\begin{equation}
F_{P1}(t,x;k,\lambda )=\left( \frac{2}{\lambda e^{t}+1}\right)
^{k}e^{tx}=\sum_{n=0}^{\infty }E_{n}^{(k)}(x;\lambda )\frac{t^{n}}{n!},
\label{Cad3}
\end{equation}%
($\left\vert t\right\vert <\pi $ when $\lambda =1$ and $\left\vert
t\right\vert <\left\vert \ln \left( -\lambda \right) \right\vert $ when $%
\lambda \neq 1$), $\lambda \in \mathbb{C}$, $k\in \mathbb{N}$ with, of
course,%
\begin{equation*}
E_{n}^{(k)}(\lambda )=E_{n}^{(k)}(0;\lambda ),
\end{equation*}%
which denotes the first kind Apostol-Euler numbers of order $k$ (\textit{cf}%
. \cite{DSkim1}, \cite{Grademir}, \cite{Ozden}, \cite{Ozden AML}, \cite{Ryoo}%
, \cite{SrivastavaBook}, \cite{Qi}). Substituting $k=\lambda =1$ into (\ref%
{Cad3}), we have the first kind Euler numbers $E_{n}=E_{n}^{(1)}(1)$, which
are defined by means of the following generating function:%
\begin{equation*}
\frac{2}{e^{t}+1}=\sum_{n=0}^{\infty }E_{n}\frac{t^{n}}{n!},
\end{equation*}%
where $\left\vert t\right\vert <\pi $ (\textit{cf}. \cite{Grademir}-\cite{Qi}%
; see also the references cited in each of these earlier works).

The Euler numbers of the second kind $E_{n}^{\ast }$ are defined by means of
the following generating function:%
\begin{equation*}
\frac{2}{e^{t}+e^{-t}}=\sum_{n=0}^{\infty }E_{n}^{\ast }\frac{t^{n}}{n!},
\end{equation*}%
where $\left\vert t\right\vert <\frac{\pi }{2}$ (\textit{cf}. \cite{Byrd}, 
\cite{Grademir}, \cite{Ozden AML}, \cite{Ryoo}, \cite{SrivastavaChoi2012}, 
\cite{Qi}; see also the references cited in each of these earlier works).

The Stirling numbers of the second kind are used in pure and applied
Mathematics. These numbers occur in combinatorics and in the theory of
partitions. The Stirling numbers of the second kind, denoted by $S_{2}(n,k)$%
, is the number of ways to partition a set of $n$ objects into $k$ groups (%
\cite{Bona}, \cite{Cigler}, \cite{Comtet}, \cite{M. Spiegel}, \cite%
{SrivastavaChoi2012}).

The $\lambda $-Stirling numbers of the second kind $S_{2}(n,v;\lambda )$ are
generalized of the Stirling number of the second kind. These numbers$\
S_{2}(n,v;\lambda )$ are defined by means of the following generating
function:%
\begin{equation}
F_{S}(t,v;\lambda )=\frac{\left( \lambda e^{t}-1\right) ^{v}}{v!}%
=\sum_{n=0}^{\infty }S_{2}(n,v;\lambda )\frac{t^{n}}{n!},  \label{SN-1}
\end{equation}%
where $v\in \mathbb{N}_{0}$ and $\lambda \in \mathbb{C}$. For further
information about these numbers, the reader should consult \cite{Luo} and (%
\cite{SimsekFPTA}, \cite{SimsekMANISA}, \cite{Srivastava2011}; see also the
references cited in each of these earlier works).

Observe that%
\begin{equation*}
S_{2}(n,v)=S_{2}(n,v;1),
\end{equation*}%
which are computing by the following formulas:%
\begin{equation*}
x^{n}=\sum_{v=0}^{n}\left( 
\begin{array}{c}
x \\ 
v%
\end{array}%
\right) v!S_{2}(n,v)
\end{equation*}%
or%
\begin{equation*}
S_{2}(n,v)=\frac{1}{v!}\sum_{j=0}^{v}\left( 
\begin{array}{c}
v \\ 
j%
\end{array}%
\right) (-1)^{j}\left( v-j\right) ^{n}
\end{equation*}%
(\textit{cf}. \cite{Grademir}-\cite{Qi}; see also the references cited in
each of these earlier works). A Recurrence relation for these numbers is
given by%
\begin{equation*}
S_{2}(n,k)=S_{2}(n-1,k-1)+kS_{2}(n-1,k),
\end{equation*}%
with%
\begin{equation*}
S_{2}(n,0)=0\text{ (}n\in \mathbb{N}\text{); }S_{2}(n,n)=1\text{ (}n\in 
\mathbb{N}\text{); }S_{2}(n,1)=1\text{ (}n\in \mathbb{N}\text{)}
\end{equation*}%
and $S_{2}(n,k)=0$ ($n<k$ or $k<0$) (\textit{cf}. \cite{Grademir}-\cite{Qi};
see also the references cited in each of these earlier works).

In \cite{SimsekFPTA}, we defined the $\lambda $-array polynomials $%
S_{v}^{n}(x;\lambda )$ by means of the following generating function:%
\begin{equation}
F_{A}(t,x,v;\lambda )=\frac{\left( \lambda e^{t}-1\right) ^{v}}{v!}%
e^{tx}=\sum_{n=0}^{\infty }S_{v}^{n}(x;\lambda )\frac{t^{n}}{n!},
\label{ARY-1}
\end{equation}
where $v\in \mathbb{N}_{0}$ and $\lambda \in \mathbb{C}$ (\textit{cf}. \cite%
{Bayad}, \cite{SimsekFPTA}).

The array polynomials $S_{v}^{n}(x)$ are defined by means of the following
generating function:%
\begin{equation}
F_{A}(t,x,v)=\frac{\left( e^{t}-1\right) ^{v}}{v!}e^{tx}=\sum_{n=0}^{\infty
}S_{v}^{n}(x)\frac{t^{n}}{n!},  \label{Ary}
\end{equation}%
(\textit{cf}. \cite{Bayad}, \cite{Chan}, \cite{SimsekFPTA}; see also the
references cited in each of these earlier works). By using the above
generating function, we have%
\begin{equation*}
S_{v}^{n}(x)=\frac{1}{v!}\sum_{j=0}^{v}(-1)^{v-j}\left( 
\begin{array}{c}
v \\ 
j%
\end{array}%
\right) \left( x+j\right) ^{n}
\end{equation*}%
with%
\begin{equation*}
S_{0}^{0}(x)=S_{n}^{n}(x)=1,S_{0}^{n}(x)=x^{n}
\end{equation*}%
and for $v>n$,%
\begin{equation*}
S_{v}^{n}(x)=0
\end{equation*}%
(\textit{cf}. \cite{Chan}, \cite{SimsekFPTA}, \cite{AM2014}; see also the
references cited in each of these earlier works).

Recently, the central factorial numbers $T(n,k)$ have been studied by many
authors. These functions have been many applications in theory of
Combinatorics and Probability. The central factorial numbers $T(n,k)$ (of
the second kind) are defined by means of the following generating function:%
\begin{equation}
F_{T}(t,k)=\frac{1}{(2k)!}\left( e^{t}+e^{-t}-2\right)
^{k}=\sum_{n=0}^{\infty }T(n,k)\frac{t^{2n}}{(2n)!}  \label{CT-1}
\end{equation}%
(\textit{cf}. \cite{Bona}, \cite{Cigler}, \cite{Comtet}, \cite{SrivastavaLiu}%
, \cite{AM2014}, \cite{M. Spiegel}; see also the references cited in each of
these earlier works).

These numbers have the following relations:%
\begin{equation*}
x^{n}=\sum_{k=0}^{n}T(n,k)x(x-1)(x-2^{2})(x-3^{2})\cdots (x-(k-1)^{2}).
\end{equation*}%
Combining the above equation with (\ref{C1}), we also have%
\begin{equation*}
T(n,k)=T(n-1,k-1)+k^{2}T(n-1,k),
\end{equation*}%
where $n\geq 1$, $k\geq 1$, $(n,k)\neq (1,1)$. For $n$, $k\in \mathbb{N}$, $%
T(0,k)=T(n,0)=0$ and $T(n,1)=1$ (\textit{cf}. \cite{Bona}, \cite{Cigler}, 
\cite{Comtet}, \cite{SrivastavaLiu}, \cite{AM2014}, \cite{M. Spiegel}).

\section{A family of new numbers $y_{1}(n,k;\protect\lambda )$}

In this section, we give generating function for the numbers $%
y_{1}(n,k;\lambda )$. We give some functional equations and differential
equations of this generating function. By using these equations, we derive
various new identities and combinatorics relations related to these numbers.
Some our observations on these numbers can be briefly expressed as follows:
the numbers $y_{1}(n,k;\lambda )$ are related to the $\lambda $-Stirling
numbers of the second kind, the central factorial numbers, the Euler numbers
of negative orders and the Golombek's identity.

It is time to give the following generating function for these numbers:%
\begin{equation}
F_{y_{1}}(t,k;\lambda )=\frac{1}{k!}\left( \lambda e^{t}+1\right)
^{k}=\sum_{n=0}^{\infty }y_{1}(n,k;\lambda )\frac{t^{n}}{n!},  \label{ay1}
\end{equation}%
where $k\in \mathbb{N}_{0}$ and $\lambda \in \mathbb{C}$.

Note that there is one generating function for each value of $k$. This
function is an analytic function.

By using (\ref{ay1}), we get%
\begin{equation*}
\sum_{n=0}^{\infty }y_{1}(n,k;\lambda )\frac{t^{n}}{n!}=\sum_{n=0}^{\infty
}\left( \frac{1}{k!}\sum_{j=0}^{k}\left( 
\begin{array}{c}
k \\ 
j%
\end{array}%
\right) j^{n}\lambda ^{j}\right) \frac{t^{n}}{n!}.
\end{equation*}%
Comparing the coefficients of $t^{n}$ on both sides of the above equation,
we arrive at the the following theorem:

\begin{theorem}
Let $n$ be a positive integer. Then we have%
\begin{equation}
y_{1}(n,k;\lambda )=\frac{1}{k!}\sum_{j=0}^{k}\left( 
\begin{array}{c}
k \\ 
j%
\end{array}%
\right) j^{n}\lambda ^{j}.  \label{ay2}
\end{equation}
\end{theorem}

We assume that $\lambda \neq 0$. For $k=0,1,2,3,4$ and we $n=0,1,2,3,4,5$
compute a few values of the numbers $y_{1}(n,k;\lambda )$ given by Equation (%
\ref{ay2}) as follows:

\begin{equation*}
\begin{tabular}{llllll}
$n\backslash k$ & $0$ & $1$ & $2$ & $3$ & $4$ \\ 
$0$ & $1$ & $\lambda $ & $\frac{1}{2}\lambda ^{2}+\lambda $ & $\frac{1}{6}%
\lambda ^{3}+\frac{1}{2}\lambda ^{2}+\frac{1}{2}\lambda $ & $\frac{1}{24}%
\lambda ^{4}+\frac{1}{6}\lambda ^{3}+\frac{1}{4}\lambda ^{2}+\frac{1}{6}%
\lambda $ \\ 
$1$ & $0$ & $\lambda $ & $\lambda ^{2}+\lambda $ & $\frac{1}{2}\lambda
^{3}+\lambda ^{2}+\frac{1}{2}\lambda $ & $\frac{1}{6}\lambda ^{4}+\frac{1}{2}%
\lambda ^{3}+\frac{1}{2}\lambda ^{2}+\frac{1}{6}\lambda $ \\ 
$2$ & $0$ & $\lambda $ & $2\lambda ^{2}+\lambda $ & $\frac{3}{2}\lambda
^{3}+2\lambda ^{2}+\frac{1}{2}\lambda $ & $\frac{2}{3}\lambda ^{4}+\frac{3}{2%
}\lambda ^{3}+\lambda ^{2}+\frac{1}{6}\lambda $ \\ 
$3$ & $0$ & $\lambda $ & $4\lambda ^{2}+\lambda $ & $\frac{9}{2}\lambda
^{3}+4\lambda ^{2}+\frac{1}{2}\lambda $ & $\frac{8}{3}\lambda ^{4}+\frac{9}{2%
}\lambda ^{3}+2\lambda ^{2}+\frac{1}{6}\lambda $ \\ 
$4$ & $0$ & $\lambda $ & $8\lambda ^{2}+\lambda $ & $\frac{27}{2}\lambda
^{3}+8\lambda +\frac{1}{2}\lambda $ & $\frac{32}{3}\lambda ^{4}+\frac{27}{2}%
\lambda ^{3}+4\lambda ^{2}+\frac{1}{6}\lambda $ \\ 
$5$ & $0$ & $\lambda $ & $16\lambda ^{2}+\lambda $ & $\frac{81}{2}\lambda
^{3}+816\lambda +\frac{1}{2}\lambda $ & $\frac{128}{3}\lambda ^{4}+\frac{81}{%
2}\lambda ^{3}+8\lambda ^{2}+\frac{1}{6}\lambda $%
\end{tabular}%
\end{equation*}%
For $k=0,1,2,\ldots ,9$ and $n=0,1,2,\ldots ,9$ we compute a few values of
the numbers $y_{1}(n,k;1)$ given by Equation (\ref{ay2}) as follows:%
\begin{equation*}
\begin{tabular}{lllllllllll}
$n\backslash k$ & $0$ & $1$ & $2$ & $3$ & $4$ & $5$ & $6$ & $7$ & $8$ & $9$
\\ 
$0$ & $1$ & $2$ & $2$ & $\frac{4}{3}$ & $\frac{2}{3}$ & $\frac{4}{15}$ & $%
\frac{4}{45}$ & $\frac{8}{315}$ & $\frac{2}{315}$ & $\frac{4}{2835}$ \\ 
$1$ & $0$ & $1$ & $2$ & $2$ & $\frac{4}{3}$ & $\frac{2}{3}$ & $\frac{4}{15}$
& $\frac{4}{45}$ & $\frac{8}{315}$ & $\frac{2}{315}$ \\ 
$2$ & $0$ & $1$ & $3$ & $4$ & $\frac{10}{3}$ & $2$ & $\frac{14}{15}$ & $%
\frac{16}{45}$ & $\frac{4}{35}$ & $\frac{2}{63}$ \\ 
$3$ & $0$ & $1$ & $5$ & $9$ & $\frac{28}{3}$ & $\frac{20}{3}$ & $\frac{18}{5}
$ & $\frac{14}{9}$ & $\frac{176}{315}$ & $\frac{6}{35}$ \\ 
$4$ & $0$ & $1$ & $9$ & $22$ & $\frac{85}{3}$ & $24$ & $\frac{224}{15}$ & $%
\frac{328}{45}$ & $\frac{102}{35}$ & $\frac{62}{63}$ \\ 
$5$ & $0$ & $1$ & $17$ & $57$ & $\frac{274}{3}$ & $\frac{275}{3}$ & $\frac{%
328}{5}$ & $\frac{1624}{45}$ & $\frac{5048}{315}$ & $\frac{208}{35}$ \\ 
$6$ & $0$ & $1$ & $33$ & $154$ & $\frac{925}{3}$ & $367$ & $\frac{4529}{15}$
& $\frac{8416}{45}$ & $\frac{3224}{35}$ & $\frac{2360}{63}$ \\ 
$7$ & $0$ & $1$ & $65$ & $429$ & $\frac{3238}{3}$ & $\frac{4580}{3}$ & $%
\frac{7223}{5}$ & $\frac{9065}{9}$ & $\frac{173216}{315}$ & $\frac{8576}{35}$
\\ 
$8$ & $0$ & $1$ & $129$ & $1222$ & $\frac{11665}{3}$ & $6554$ & $\frac{107114%
}{15}$ & $\frac{252268}{45}$ & $\frac{118717}{35}$ & $\frac{104288}{63}$ \\ 
$9$ & $0$ & $1$ & $257$ & $3537$ & $\frac{42994}{3}$ & $\frac{86645}{3}$ & $%
\frac{181458}{5}$ & $\frac{1444534}{45}$ & $\frac{6781748}{315}$ & $\frac{%
402723}{35}$%
\end{tabular}%
\end{equation*}%
Some special values of $y_{1}(n,k;\lambda )$ are given as follows:%
\begin{equation*}
y_{1}(0,k;\lambda )=\frac{1}{k!}(\lambda +1)^{k},
\end{equation*}%
\begin{equation*}
y_{1}(n,0;\lambda )=0,
\end{equation*}%
and%
\begin{equation*}
y_{1}(n,1;\lambda )=\lambda .
\end{equation*}%
By using (\ref{ay1}), we derive the following functional equation%
\begin{equation*}
\lambda ^{k}e^{kt}=\sum_{l=0}^{k}(-1)^{k-l}\left( 
\begin{array}{c}
k \\ 
l%
\end{array}%
\right) l!F_{y_{1}}(t,l;\lambda ).
\end{equation*}%
Combining (\ref{ay1}) with the above equation, we get%
\begin{equation*}
\lambda ^{k}\sum_{n=0}^{\infty }\frac{\left( kt\right) ^{n}}{n!}%
=\sum_{n=0}^{\infty }\left( \sum_{l=0}^{k}(-1)^{k-l}\left( 
\begin{array}{c}
k \\ 
l%
\end{array}%
\right) l!y_{1}(n,l;\lambda )\right) \frac{t^{n}}{n!}.
\end{equation*}%
Comparing the coefficients of $\frac{t^{n}}{n!}$ on both sides of the above
equation, we arrive at the following theorem:

\begin{theorem}
\begin{equation*}
k^{n}\lambda ^{k}=\sum_{l=0}^{k}(-1)^{k-l}\left( 
\begin{array}{c}
k \\ 
l%
\end{array}%
\right) l!y_{1}(n,l;\lambda ).
\end{equation*}
\end{theorem}

We give a relationship between the numbers $y_{1}(n,k;\lambda )$\ and the $%
\lambda $-Stirling numbers of the second kind by the following theorem:

\begin{theorem}
\begin{equation*}
S_{2}(n,k;\lambda ^{2})=\frac{k!}{2^{n}}\sum_{l=0}^{n}\left( 
\begin{array}{c}
n \\ 
l%
\end{array}%
\right) S_{2}(l,k;\lambda )y_{1}(n-l,k;\lambda ).
\end{equation*}
\end{theorem}

\begin{proof}
By using (\ref{SN-1}) and (\ref{ay1}), we derive the following functional
equation:%
\begin{equation*}
F_{S}(2t,v;\lambda ^{2})=k!F_{S}(t,v;\lambda )F_{y_{1}}(t,k;\lambda ).
\end{equation*}%
From this equation, we have%
\begin{equation*}
\sum_{n=0}^{\infty }2^{n}S\left( n,v;\lambda ^{2}\right) \frac{t^{n}}{n!}%
=\sum_{n=0}^{\infty }S\left( n,v;\lambda \right) \frac{t^{n}}{n!}%
\sum_{n=0}^{\infty }y_{1}(n,k;\lambda )\frac{t^{n}}{n!}.
\end{equation*}%
Therefore%
\begin{equation*}
\sum_{n=0}^{\infty }2^{n}S\left( n,v;\lambda ^{2}\right) \frac{t^{n}}{n!}%
=\sum_{n=0}^{\infty }\left( k!\sum_{l=0}^{n}\left( 
\begin{array}{c}
n \\ 
l%
\end{array}%
\right) S_{2}(l,k;\lambda )y_{1}(n-l,k;\lambda )\right) \frac{t^{n}}{n!}.
\end{equation*}%
Comparing the coefficients of $\frac{t^{n}}{n!}$ on both sides of the above
equation, we arrive at the desired result.
\end{proof}

A relationship between the numbers $y_{1}(n,k;\lambda )$, $S_{2}(n,k;\lambda
^{3})$ and the array polynomials $S_{k}^{n}(x;\lambda )$ is given by the
following theorem:

\begin{theorem}
\begin{equation*}
S_{2}(n,k;\lambda ^{3})=\sum_{l=0}^{n}\sum_{j=0}^{k}\left( 
\begin{array}{c}
n \\ 
l%
\end{array}%
\right) \left( 
\begin{array}{c}
k \\ 
j%
\end{array}%
\right) \frac{\lambda ^{2k-2j}j!}{3^{n}}y_{1}(l,j;\lambda
)S_{k}^{n-l}(2k-2j;\lambda ).
\end{equation*}
\end{theorem}

\begin{proof}
If we combine (\ref{SN-1}), (\ref{ARY-1}) and (\ref{ay1}), we get%
\begin{equation*}
F_{S}(3t,k;\lambda ^{3})=\sum_{j=0}^{k}\frac{k!}{(k-j)!}\lambda
^{2k-2j}F_{A}(t,2k-2j,k;\lambda )F_{y_{1}}(t,j;\lambda ).
\end{equation*}%
By using the above functional equation, we obtain%
\begin{equation*}
\sum_{n=0}^{\infty }3^{n}S\left( n,v;\lambda ^{3}\right) \frac{t^{n}}{n!}%
=\sum_{j=0}^{k}\frac{k!}{(k-j)!}\lambda ^{2k-2j}\sum_{n=0}^{\infty
}S_{k}^{n}\left( n,2k-2j;\lambda \right) \frac{t^{n}}{n!}\sum_{n=0}^{\infty
}y_{1}(n,j;\lambda )\frac{t^{n}}{n!}.
\end{equation*}%
Therefore%
\begin{equation*}
\sum_{n=0}^{\infty }3^{n}S\left( n,v;\lambda ^{3}\right) \frac{t^{n}}{n!}%
=\sum_{n=0}^{\infty }\sum_{l=0}^{n}\sum_{j=0}^{k}\left( 
\begin{array}{c}
n \\ 
l%
\end{array}%
\right) \left( 
\begin{array}{c}
k \\ 
j%
\end{array}%
\right) j!\lambda ^{2k-2j}y_{1}(l,j;\lambda )S_{k}^{n-l}(2k-2j;\lambda )%
\frac{t^{n}}{n!}.
\end{equation*}%
Comparing the coefficients of $\frac{t^{n}}{n!}$ on both sides of the above
equation, we arrive at the desired result.
\end{proof}

There are many combinatorics and analysis applications for (\ref{ay2}). By
substituting $\lambda =1$ into (\ref{ay2}), then we set%
\begin{equation}
B(n,k)=k!y_{1}(n,k;1).  \label{CC2}
\end{equation}%
In \cite{golombek}, Golombek gave the following formula for (\ref{ay2}):%
\begin{equation*}
B(n,k)=\frac{d^{n}}{dt^{n}}\left( e^{t}+1\right) ^{k}\left\vert
_{t=0}\right. .
\end{equation*}

\begin{remark}
If we substitute $\lambda =-1$ into (\ref{ay2}), then we get the Stirling
numbers of the second kind:%
\begin{equation*}
S_{2}(n,k)=(-1)^{k}y_{1}(n,k;-1)=\frac{1}{k!}\sum_{j=0}^{k}(-1)^{k-j}\left( 
\begin{array}{c}
k \\ 
j%
\end{array}%
\right) j^{n}
\end{equation*}%
(\textit{cf}. \cite{Bona}-\cite{SrivastavaBook}).
\end{remark}

\begin{remark}
We set%
\begin{equation*}
B(n,k)=k!y_{1}(n,k;1)=\frac{d^{n}}{dt^{n}}\left( e^{t}+1\right)
^{k}\left\vert _{t=0}\right. .
\end{equation*}%
These numbers are related to the following numbers:%
\begin{equation*}
a_{k}2^{k}
\end{equation*}%
where the sequence $a_{k}$ is a positive integer depend on $k$.
Consequently, in the work of Spivey \cite[Identity 8-Identity 10]{Spevy}, we
see that%
\begin{equation*}
B(0,k)=2^{k},
\end{equation*}%
\begin{equation*}
B(1,k)=k2^{k-1},
\end{equation*}%
\begin{equation*}
B(2,k)=k(k+1)2^{k-2},
\end{equation*}%
see also \cite[P. 56, Exercise 21]{Bona} and \cite[p. 117]{Charamb}.
\end{remark}

\begin{remark}
In \cite[Identity 12.]{Spevy}, Spivey also proved the following novel
identity by the falling factorial method:%
\begin{equation}
B(m,n)=\sum_{j=0}^{n}\left( 
\begin{array}{c}
n \\ 
j%
\end{array}%
\right) j!2^{n-j}S_{2}(m,j).  \label{Bs-1}
\end{equation}
\end{remark}

The numbers $B(0,k)$ are given by means of the following well-known
generating function: Let $\left\vert x\right\vert <\frac{1}{2}$, we have%
\begin{equation*}
\sum_{k=0}^{\infty }B(0,k)x^{k}=\frac{1}{1-2x}.
\end{equation*}%
The numbers $B(1,k)$ are given by means of the following well-known
generating function: Let $\left\vert x\right\vert <\frac{1}{2}$, we have%
\begin{equation*}
\sum_{k=1}^{\infty }B(1,k)x^{k}=\frac{x}{\left( 1-2x\right) ^{2}}.
\end{equation*}

\begin{remark}
In work of Boyadzhiev \cite[p.4, Eq-(7)]{Boyadzhiev}, we see that%
\begin{equation*}
\sum_{j=0}^{k}\left( 
\begin{array}{c}
k \\ 
j%
\end{array}%
\right) j^{n}x^{j}=\sum_{j=0}^{n}\left( 
\begin{array}{c}
n \\ 
j%
\end{array}%
\right) j!2^{n-j}S_{2}(m,j)x^{j}(1+x)^{n-j}.
\end{equation*}%
Substituting $x=1$ into the above equation, we arrive at (\ref{Bs-1}).
\end{remark}

\begin{theorem}
Let $d$ be a positive integer and $m_{0},m_{1},m_{2},...,m_{d}\in \mathbb{Q}$%
. Let $m_{0}\neq 0$. Thus we have%
\begin{equation}
\sum_{v=0}^{d-1}m_{v}B(d-v,k)=2^{k-d}\left( 
\begin{array}{c}
k \\ 
d%
\end{array}%
\right) .  \label{CB1}
\end{equation}
\end{theorem}

\begin{proof}
It is well-known that%
\begin{equation*}
(1+x)^{k}=\sum_{j=0}^{k}\left( 
\begin{array}{c}
k \\ 
j%
\end{array}%
\right) x^{j}.
\end{equation*}%
Taking the $k^{th}$ derivative, with respect to $x$, we obtain%
\begin{equation}
\left( 
\begin{array}{c}
k \\ 
d%
\end{array}%
\right) (1+x)^{n-d}=\sum_{j=0}^{k}\left( 
\begin{array}{c}
k \\ 
j%
\end{array}%
\right) \left( 
\begin{array}{c}
j \\ 
d%
\end{array}%
\right) x^{j-d}.  \label{Ci7}
\end{equation}%
Substituting $x=1$ into the above equation, we get%
\begin{equation}
2^{n-k}\left( 
\begin{array}{c}
k \\ 
d%
\end{array}%
\right) =\sum_{j=0}^{k}\left( 
\begin{array}{c}
k \\ 
j%
\end{array}%
\right) \left( 
\begin{array}{c}
j \\ 
d%
\end{array}%
\right) .  \label{Ci8}
\end{equation}%
In our work of \cite{mmas2015}, we know that%
\begin{equation*}
\left( 
\begin{array}{c}
j \\ 
d%
\end{array}%
\right) =m_{0}j^{d}+m_{1}j^{d-1}+\cdots +m_{d-1}j,
\end{equation*}%
where $m_{0},m_{1},\ldots ,m_{d-1}\in \mathbb{Q}$. Therefore%
\begin{equation*}
2^{k-d}\left( 
\begin{array}{c}
k \\ 
d%
\end{array}%
\right) =\sum_{j=0}^{k}\left( 
\begin{array}{c}
k \\ 
j%
\end{array}%
\right) \left( m_{0}j^{d}+m_{1}j^{d-1}+\cdots +m_{d-1}j\right) .
\end{equation*}%
Thus we get%
\begin{equation*}
2^{k-d}\left( 
\begin{array}{c}
k \\ 
d%
\end{array}%
\right) =\sum_{v=0}^{d-1}m_{v}\sum_{j=0}^{k}\left( 
\begin{array}{c}
k \\ 
j%
\end{array}%
\right) j^{d-v}.
\end{equation*}%
Combining (\ref{CC2}) with the above equation, we have%
\begin{equation*}
2^{k-d}\left( 
\begin{array}{c}
k \\ 
d%
\end{array}%
\right) =\sum_{v=0}^{d-1}m_{v}B(d-v,k).
\end{equation*}%
Thus proof of theorem is completed.
\end{proof}

There are many combinatorial arguments of (\ref{Ci7}).That is, if we
substitute $d=3$ and $4$ into (\ref{Ci7}), then we compute $B(3,k)$ and $%
B(4,k)$, respectively, as follows:%
\begin{equation*}
B(3,k)=k^{2}(k+3)2^{k-3}
\end{equation*}%
and%
\begin{equation*}
B(4,k)=k(k^{3}+6k^{2}+3k-2)2^{k-4}.
\end{equation*}

By using (\ref{CB1}), we derive the following result:%
\begin{equation*}
B(d,k)=\frac{2^{k-d}}{m_{0}}\left( 
\begin{array}{c}
k \\ 
d%
\end{array}%
\right) -\sum_{v=1}^{d-1}\frac{m_{v}}{m_{0}}B(d-v,k).
\end{equation*}

Therefore, we conjecture that 
\begin{equation*}
B(d,k)=(k^{d}+x_{1}k^{d-1}+x_{2}k^{d-2}+\cdots
++x_{d-2}k^{2}+x_{d-1}k)2^{k-d},
\end{equation*}%
where $x_{1},x_{2},\ldots ,x_{d-1}$, $d$ are positive integers.
Consequently, we arrive at the following open questions:

1-How can we compute the coefficients $x_{1},x_{2},\ldots ,x_{d-1}?$

2-We assume that for $\left\vert x\right\vert <r$%
\begin{equation*}
\sum_{k=1}^{\infty }B(d,k)x^{k}=f_{d}(x).
\end{equation*}

Is it possible to find $f_{d}(x)$ function?

\subsection{Recurrence relation and some identities for the numbers $%
y_{1}(n,k;\protect\lambda )$}

Here, by applying derivative operator to the generating functions (\ref{ay1}%
), we give a recurrence relation and other formulas for the numbers $%
y_{1}(n,k;\lambda )$.

\begin{theorem}
Let $k$ be a positive integer. Then we have%
\begin{equation*}
y_{1}(n+1,k;\lambda )=ky_{1}(n,k;\lambda )-y_{1}(n,k-1;\lambda ).
\end{equation*}
\end{theorem}

\begin{proof}
Taking derivative of (\ref{ay1}), with respect to $t$, we obtain the
following partial differential equation:%
\begin{equation*}
\frac{\partial }{\partial t}F_{y_{1}}(t,k;\lambda )=kF_{y_{1}}(t,k;\lambda
)-F_{y_{1}}(t,k-1;\lambda ).
\end{equation*}%
Combining (\ref{ay1}) with the above equation, we get%
\begin{equation*}
\sum_{n=1}^{\infty }y_{1}(n,k;\lambda )\frac{t^{n-1}}{\left( n-1\right) !}%
=k\sum_{n=0}^{\infty }y_{1}(n,k;\lambda )\frac{t^{n}}{n!}-\sum_{n=0}^{\infty
}y_{1}(n,k-1;\lambda )\frac{t^{n}}{n!}.
\end{equation*}%
After some elementary calculation, comparing the coefficients of $\frac{t^{n}%
}{n!}$ on both sides of the above equation, we arrive at the desired result.
\end{proof}

\begin{theorem}
Let $k$ be a positive integer. Then we have%
\begin{equation*}
\frac{\partial }{\partial \lambda }y_{1}(n,k;\lambda )=\sum_{j=0}^{n}\left( 
\begin{array}{c}
n \\ 
j%
\end{array}%
\right) y_{1}(j,k-1;\lambda ).
\end{equation*}
\end{theorem}

\begin{proof}
Taking derivative of (\ref{ay1}), with respect to $\lambda $, we obtain the
following partial differential equation:%
\begin{equation}
\frac{\partial }{\partial \lambda }F_{y_{1}}(t,k;\lambda
)=e^{t}F_{y_{1}}(t,k-1;\lambda ).  \label{D.y1}
\end{equation}%
Combining (\ref{ay1}) with the above equation, we get%
\begin{equation*}
\sum_{n=0}^{\infty }\frac{\partial }{\partial \lambda }y_{1}(n,k;\lambda )%
\frac{t^{n}}{n!}=\sum_{n=0}^{\infty }\sum_{j=0}^{n}\left( 
\begin{array}{c}
n \\ 
j%
\end{array}%
\right) y_{1}(j,k-1;\lambda )\frac{t^{n}}{n!}.
\end{equation*}%
After some elementary calculation, comparing the coefficients of $\frac{t^{n}%
}{n!}$ on both sides of the above equation, we arrive at the desired result.
\end{proof}

\begin{theorem}
Let $k$ be a positive integer. Then we have%
\begin{equation*}
\lambda \frac{\partial }{\partial \lambda }y_{1}(n,k;\lambda
)=ky_{1}(n,k;\lambda )-y_{1}(n,k-1;\lambda ).
\end{equation*}
\end{theorem}

\begin{proof}
By using (\ref{D.y1}), we obtain the following partial differential equation:%
\begin{equation*}
\lambda \frac{\partial }{\partial \lambda }F_{y_{1}}(t,k;\lambda
)=kF_{y_{1}}(t,k;\lambda )-F_{y_{1}}(t,k-1;\lambda ).
\end{equation*}%
Combining (\ref{ay1}) with the above equation, we get%
\begin{equation*}
\sum_{n=0}^{\infty }\frac{\partial }{\partial \lambda }y_{1}(n,k;\lambda )%
\frac{t^{n}}{n!}=\sum_{n=0}^{\infty }ky_{1}(n,k;\lambda )\frac{t^{n}}{n!}%
-\sum_{n=0}^{\infty }y_{1}(j,k-1;\lambda )\frac{t^{n}}{n!}.
\end{equation*}%
Comparing the coefficients of $\frac{t^{n}}{n!}$ on both sides of the above
equation, we arrive at the desired result.
\end{proof}

\section{A family of new numbers $y_{2}(n,k;\protect\lambda )$}

In this section, we define a family of new numbers $y_{2}(n,k;\lambda )$ by
means of the following generating function:%
\begin{equation}
F_{y_{2}}(t,k;\lambda )=\frac{1}{(2k)!}\left( \lambda e^{t}+\lambda
^{-1}e^{-t}+2\right) ^{k}=\sum_{n=0}^{\infty }y_{2}(n,k;\lambda )\frac{t^{n}%
}{n!},  \label{C1}
\end{equation}%
where $k\in \mathbb{N}_{0}$ and $\lambda \in \mathbb{C}$.

Note that there is one generating function for each value of $k$.

In this section, by using (\ref{C1}) with their functional equation, we
derive various identities and relations including our new numbers, the
Fibonacci numbers, the Lucas numbers, the Stirling numbers and the central
factorial numbers. 

By using (\ref{C1}), we get the following explicit formula for the numbers $%
y_{2}(n,k;\lambda )$:

\begin{theorem}
\begin{equation}
y_{2}(n,k;\lambda )=\frac{1}{\left( 2k\right) !}\sum_{j=0}^{k}\left( 
\begin{array}{c}
k \\ 
j%
\end{array}%
\right) 2^{k-j}\sum_{l=0}^{j}\left( 
\begin{array}{c}
j \\ 
l%
\end{array}%
\right) \left( 2l-j\right) ^{n}\lambda ^{2l-j}.  \label{CCC3}
\end{equation}
\end{theorem}

\begin{proof}
By (\ref{C1}), we have%
\begin{equation*}
\sum_{n=0}^{\infty }y_{2}(n,k;\lambda )\frac{t^{n}}{n!}=\sum_{n=0}^{\infty
}\left( \frac{1}{\left( 2k\right) !}\sum_{j=0}^{k}\left( 
\begin{array}{c}
k \\ 
j%
\end{array}%
\right) 2^{k-j}\sum_{l=0}^{j}\left( 
\begin{array}{c}
j \\ 
l%
\end{array}%
\right) \left( 2l-j\right) ^{n}\lambda ^{2l-j}\right) \frac{t^{n}}{n!}.
\end{equation*}%
Comparing the coefficients of $\frac{t^{n}}{n!}$ on both sides of the above
equation, we arrive at the desired result.
\end{proof}

For $k=0,1,2,3$ and we $n=0,1,2,3,4,5$ compute a few values of the numbers $%
y_{2}(n,k;\lambda )$ given by Equation (\ref{CCC3}) as follows:

\begin{equation*}
\begin{tabular}{lllll}
$n\backslash k$ & $0$ & $1$ & $2$ & $3$ \\ 
$0$ & $1$ & $\frac{1}{2\lambda }+\frac{\lambda }{2}$ & $\frac{\lambda
^{2}+4\lambda }{24}+\frac{4\lambda +1}{24\lambda ^{2}}$ & $\frac{\lambda
^{3}+6\lambda ^{2}}{720}+\frac{\lambda }{48}+\frac{1}{48\lambda }+\frac{%
6\lambda +1}{720\lambda ^{3}}$ \\ 
$1$ & $0$ & $\frac{\lambda }{2}-\frac{1}{2\lambda }$ & $\frac{\lambda
^{2}+2\lambda }{12}-\frac{2\lambda +1}{6\lambda ^{2}}$ & $\frac{\lambda
^{3}+4\lambda ^{2}}{240}+\frac{\lambda }{28}-\frac{1}{48\lambda }-\frac{%
4\lambda +1}{240\lambda ^{3}}$ \\ 
$2$ & $0$ & $\frac{\lambda }{2}+\frac{1}{2\lambda }$ & $\frac{\lambda
^{2}+\lambda }{6}+\frac{\lambda +1}{6\lambda ^{2}}$ & $\frac{\lambda ^{3}}{80%
}+\frac{\lambda ^{2}}{30}+\frac{\lambda }{48}+\frac{1}{48\lambda }+\frac{1}{%
30\lambda ^{2}}+\frac{1}{80\lambda ^{3}}$ \\ 
$3$ & $0$ & $\frac{\lambda }{2}-\frac{1}{2\lambda }$ & $\frac{2\lambda
^{2}+\lambda }{6}-\frac{\lambda +2}{6\lambda ^{2}}$ & $\frac{3\lambda ^{3}}{%
80}+\frac{\lambda ^{2}}{15}+\frac{\lambda }{48}-\frac{1}{48\lambda }-\frac{1%
}{15\lambda ^{2}}-\frac{3}{80\lambda ^{3}}$ \\ 
$4$ & $0$ & $\frac{\lambda }{2}+\frac{1}{2\lambda }$ & $\frac{2\lambda
^{2}+\lambda }{3}+\frac{\lambda +4}{6\lambda ^{2}}$ & $\frac{9\lambda ^{3}}{%
80}+\frac{2\lambda ^{2}}{15}+\frac{\lambda }{48}+\frac{1}{48\lambda }+\frac{2%
}{15\lambda ^{2}}+\frac{9}{80\lambda ^{3}}$ \\ 
$5$ & $0$ & $\frac{\lambda }{2}-\frac{1}{2\lambda }$ & $\frac{8\lambda
^{2}+\lambda }{6}-\frac{\lambda +8}{6\lambda ^{2}}$ & $\frac{27\lambda ^{3}}{%
80}+\frac{4\lambda ^{2}}{15}+\frac{\lambda }{48}-\frac{1}{48\lambda }-\frac{4%
}{15\lambda ^{2}}-\frac{27}{80\lambda ^{3}}$%
\end{tabular}%
\end{equation*}%
By using (\ref{ay1}) and (\ref{C1}), we get the following functional
equation:%
\begin{equation*}
F_{y_{2}}(t,k;\lambda )=\frac{k!}{(2k)!}\sum_{j=0}^{k}F_{y_{1}}(t,j;\lambda
)F_{y_{1}}\left( -t,k-j;\lambda ^{-1}\right) .
\end{equation*}%
By combining (\ref{ay1}) and (\ref{C1}) with the above equation, we obtain%
\begin{equation*}
\sum_{n=0}^{\infty }y_{2}(n,k;\lambda )\frac{t^{n}}{n!}\frac{k!}{(2k)!}%
\sum_{j=0}^{k}\left( \sum_{n=0}^{\infty }y_{1}(n,j;\lambda )\frac{t^{n}}{n!}%
\sum_{n=0}^{\infty }(-1)^{n}y_{1}(n,k-j;\lambda ^{-1})\frac{t^{n}}{n!}%
\right) .
\end{equation*}%
Therefore%
\begin{equation*}
\sum_{n=0}^{\infty }y_{2}(n,k;\lambda )\frac{t^{n}}{n!}=\frac{k!}{(2k)!}%
\sum_{n=0}^{\infty }\sum_{j=0}^{k}\sum_{l=0}^{n}(-1)^{n-l}\left( 
\begin{array}{c}
n \\ 
l%
\end{array}%
\right) y_{1}(l,j;\lambda )y_{1}(n-l,k-j;\lambda ^{-1})\frac{t^{n}}{n!}.
\end{equation*}%
Comparing the coefficients of $\frac{t^{n}}{n!}$ on both sides of the above
equation, the numbers $y_{2}(n,k;\lambda )$ is given in terms of the numbers 
$y_{1}(n,k;\lambda )$ by the following theorem:

\begin{theorem}
\begin{equation}
y_{2}(n,k;\lambda )=\frac{k!}{(2k)!}\sum_{j=0}^{k}\sum_{l=0}^{n}(-1)^{n-l}%
\left( 
\begin{array}{c}
n \\ 
l%
\end{array}%
\right) y_{1}(l,j;\lambda )y_{1}(n-l,k-j;\lambda ^{-1}).  \label{Alg-2}
\end{equation}
\end{theorem}

\begin{theorem}
\label{TheoremG copy(2)}%
\begin{equation*}
y_{1}(n,2k;\lambda )=\lambda ^{k}\sum_{j=0}^{n}\left( 
\begin{array}{c}
n \\ 
j%
\end{array}%
\right) k^{n-j}y_{2}(j,k;\lambda ).
\end{equation*}
\end{theorem}

\begin{proof}
By using (\ref{ay1}) and (\ref{C1}), we get the following functional
equation:%
\begin{equation*}
\lambda ^{k}e^{kt}F_{y_{2}}(t,k;\lambda )=F_{y_{1}}(t,2k;\lambda ).
\end{equation*}%
From the above functional equation, we obtain%
\begin{equation*}
\sum_{n=0}^{\infty }y_{1}(n,2k;\lambda )\frac{t^{n}}{n!}=\lambda
^{k}\sum_{n=0}^{\infty }\frac{(kt)^{n}}{n!}\sum_{n=0}^{\infty
}y_{2}(n,k;\lambda )\frac{t^{n}}{n!}.
\end{equation*}%
Therefore%
\begin{equation*}
\sum_{n=0}^{\infty }y_{1}(n,2k;\lambda )\frac{t^{n}}{n!}=\sum_{n=0}^{\infty
}\left( \lambda ^{k}\sum_{j=0}^{n}\left( 
\begin{array}{c}
n \\ 
j%
\end{array}%
\right) k^{n-j}y_{2}(j,k;\lambda )\right) \frac{t^{n}}{n!}.
\end{equation*}%
Comparing the coefficients of $\frac{t^{n}}{n!}$ on both sides of the above
equation, we arrive at the desired result.
\end{proof}

By substituting $\lambda =1$ into (\ref{C1}), we have%
\begin{equation*}
F_{y_{2}}(t,k)=\frac{1}{(2k)!}\left( e^{t}+e^{-t}+2\right) ^{k}.
\end{equation*}%
The function $F_{y_{2}}(t,k)$ is an even function. Consequently, we get the
following result:%
\begin{equation*}
y_{2}(2n+1,k;1)=0.
\end{equation*}%
Thus, we get%
\begin{equation}
F_{y_{2}}(t,k;1)=\sum_{n=0}^{\infty }y_{2}(2n,k;1)\frac{t^{2n}}{(2n)!}.
\label{CC1}
\end{equation}%
By using (\ref{CC1}), we give the following explicit formula for the numbers 
$y_{2}(n,k)$ $(=y_{2}(n,k;1))$:

\begin{corollary}
\begin{equation}
y_{2}(n,k)=\frac{1}{\left( 2k\right) !}\sum_{j=0}^{k}\left( 
\begin{array}{c}
k \\ 
j%
\end{array}%
\right) 2^{k-j}\sum_{l=0}^{j}\left( 
\begin{array}{c}
j \\ 
l%
\end{array}%
\right) \left( 2l-j\right) ^{n}.  \label{CCa3}
\end{equation}
\end{corollary}

For $k=0,1,2,\ldots ,9$, we compute a few values of the numbers $y_{2}(n,k)$
given by Equation (\ref{CCa3}) as follows:%
\begin{equation*}
y_{2}(0,0)=1,
\end{equation*}%
\begin{equation*}
y_{2}(n,0)=0,\text{ (}n\in \mathbb{N}\text{)}
\end{equation*}%
\begin{equation*}
y_{2}(n,1)=(-1)^{n}+1,
\end{equation*}%
\begin{equation*}
y_{2}(n,2)=\frac{\left( (-1)^{n}+1\right) }{6}+\frac{2^{n-1}-(-2)^{n-1}}{3},
\end{equation*}%
\begin{equation*}
y_{2}(n,3)=\frac{\left( (-1)^{n}+1\right) }{24}+\frac{2^{n-2}+(-2)^{n-2}}{15}%
+\frac{(-3)^{n-2}+3^{n-2}}{10},
\end{equation*}%
\begin{eqnarray*}
y_{2}(n,4) &=&\frac{13\left( (-1)^{n}+1\right) }{5040}+\frac{%
2^{n-1}-(-2)^{n-1})}{315}+\frac{4^{n-1}-(-4)^{n-1}}{630}+\frac{%
3^{n-2}+(-3)^{n-2}}{140} \\
&&+\frac{2^{n-4}+(-2)^{n-4}}{105},
\end{eqnarray*}

\begin{eqnarray*}
y_{2}(n,5) &=&\frac{19\left( (-1)^{n}+1\right) }{120960}+\frac{%
5^{n-2}+(-5)^{n-2}}{4536}+\frac{4^{n-2}+(-4)^{n-2}}{2835}+\frac{%
2^{n-2}+(-2)^{n-2}}{2835} \\
&&+\frac{3^{n-3}-(-3)^{n-3}}{560}+\frac{2^{n-5}-(-2)^{n-5}}{2835},
\end{eqnarray*}

\begin{eqnarray*}
y_{2}(n,6) &=&\frac{67\left( (-1)^{n}+1\right) }{13305600}+\frac{%
5^{n-2}+(-5)^{n-2}}{99792}+2\frac{4^{n-2}+(-4)^{n-2}}{155925}+\frac{%
4^{n-3}-(-4)^{n-3}}{31185} \\
&&+\frac{2^{n-3}-(-2)^{n-3}}{31185}+\frac{3^{n-4}+(-3)^{n-4}}{6160}+2\frac{%
6^{n-5}-(-6)^{n-5}}{1925}+\frac{3^{n-5}-(-3)^{n-5}}{12320} \\
&&+\frac{2^{n-8}+(-2)^{n-8}}{31185},
\end{eqnarray*}

\begin{eqnarray*}
y_{2}(n,7) &=&\frac{41\left( (-1)^{n}+1\right) }{296524800}+\frac{%
7^{n-2}+(-7)^{n-2}}{13899600}+\frac{5^{n-2}+(-5)^{n-2}}{2223936}+\frac{%
4^{n-2}+(-4)^{n-2}}{2027025} \\
&&+\frac{3^{n-2}+(-3)^{n-2}}{1601600}+\frac{4^{n-3}-(-4)^{n-3}}{2432430}+%
\frac{2^{n-3}-(-2)^{n-3}}{1216215}+\frac{6^{n-5}-(-6)^{n-5}}{25025} \\
&&+\frac{3^{n-5}-(-3)^{n-5}}{1281280}+\frac{2^{n-9}-(-2)^{n-9}}{2027025},
\end{eqnarray*}

\begin{eqnarray*}
y_{2}(n,8) &=&\frac{53\left( (-1)^{n}+1\right) }{20118067200}+\frac{%
7^{n-2}+(-7)^{n-2}}{416988000}+\frac{5^{n-2}+(-5)^{n-2}}{93405312}+\frac{%
4^{n-2}+(-4)^{n-2}}{182432250} \\
&&+4\frac{8^{n-3}}{638512875}-\frac{4^{n-3}-(-4)^{n-3}}{30405375}-\frac{%
4(-8)^{n-3}}{638512875}+\frac{2^{n-4}+(-2)^{n-4}}{91216125} \\
&&+\frac{4^{n-5}-(-4)^{n-5}}{18243225}+23\frac{3^{n-5}}{64064000}-\frac{%
(-3)^{n-5}}{4004000}+\frac{3\left( 6^{n-6}+(-6)^{n-6}\right) }{350350} \\
&&-\frac{(-3)^{n-5}}{9152000}+\frac{2^{n-7}-(-2)^{n-7}}{6081075}+\frac{%
2^{n-8}-(-2)^{n-8}}{18243225}+\frac{2^{n-11}-(-2)^{n-11}}{91216125},
\end{eqnarray*}

\begin{eqnarray*}
y_{2}(n,9) &=&\frac{14897\left( (-1)^{n}+1\right) }{355687428096000}+\frac{%
7^{n-2}+(-7)^{n-2}}{14177592000}+19\frac{5^{n-2}+(-5)^{n-2}}{88921857024}+2%
\frac{8^{n-3}-(-8)^{n-3}}{10854718875} \\
&&+\frac{4^{n-3}-(-4)^{n-3}}{1550674125}+\frac{9^{n-4}+(-9)^{n-4}}{1905904000%
}+2\frac{4^{n-4}+(-4)^{n-4}}{1550674125}+\frac{3^{n-5}-(-3)^{n-5}}{272272000}
\\
&&+\frac{2^{n-5}-(-2)^{n-5}}{1550674125}+2\frac{6^{n-6}}{14889875}+\frac{%
3^{n-6}+(-3)^{n-6}}{435635200}+\frac{2\left( 4^{n-6}+(-4)^{n-6}\right) }{%
1550674125} \\
&&+\frac{2(-6)^{n-6}}{14889875}+\frac{6^{n-7}-(-6)^{n-7}}{4254250}+\frac{%
3^{n-7}-(-3)^{n-7}}{67020800}+\frac{2^{n-8}+(-2)^{n-8}}{310134825} \\
&&+\frac{2^{n-9}-(-2)^{n-9}}{1550674125}+\frac{2^{n-12}+(-2)^{n-12}}{%
10854718875}.
\end{eqnarray*}

For $k=0$, we have%
\begin{eqnarray*}
y_{2}(0,0) &=&1, \\
y_{2}(0,1) &=&2, \\
y_{2}(0,2) &=&\frac{2}{3}, \\
y_{2}(0,3) &=&\frac{5}{36}, \\
y_{2}(0,4) &=&\frac{63}{5292},
\end{eqnarray*}%
and for $k=0,1,2,\ldots ,9$ and $n=1,2,\ldots ,9$, we compute a few values
of the numbers $y_{2}(n,k)$ given by Equation (\ref{CCa3}) as follows:

\begin{equation*}
\begin{tabular}{lllllllllll}
$n\backslash k$ & $0$ & $1$ & $2$ & $3$ & $4$ & $5$ & $6$ & $7$ & $8$ & $9$
\\ 
$1$ & $0$ & $0$ & $0$ & $0$ & $0$ & $0$ & $0$ & $0$ & $0$ & $0$ \\ 
$2$ & $0$ & $1$ & $\frac{2}{3}$ & $\frac{2}{15}$ & $\frac{4}{315}$ & $\frac{2%
}{2835}$ & $\frac{4}{155925}$ & $\frac{4}{6081075}$ & $\frac{8}{638512875}$
& $\frac{2}{10854718875}$ \\ 
$3$ & $0$ & $0$ & $0$ & $0$ & $0$ & $0$ & $0$ & $0$ & $0$ & $0$ \\ 
$4$ & $0$ & $1$ & $\frac{5}{3}$ & $\frac{8}{15}$ & $\frac{22}{315}$ & $\frac{%
2}{405}$ & $\frac{34}{155925}$ & $\frac{8}{1216215}$ & $\frac{92}{638512875}$
& $\frac{2}{834978375}$ \\ 
$5$ & $0$ & $0$ & $0$ & $0$ & $0$ & $0$ & $0$ & $0$ & $0$ & $0$ \\ 
$6$ & $0$ & $1$ & $\frac{17}{3}$ & $\frac{47}{15}$ & $\frac{184}{315}$ & $%
\frac{152}{2835}$ & $\frac{454}{155925}$ & $\frac{634}{6081075}$ & $\frac{%
1688}{638512875}$ & $\frac{542}{10854718875}$ \\ 
$7$ & $0$ & $0$ & $0$ & $0$ & $0$ & $0$ & $0$ & $0$ & $0$ & $0$ \\ 
$8$ & $0$ & $1$ & $\frac{65}{3}$ & $\frac{338}{15}$ & $\frac{1957}{315}$ & $%
\frac{2144}{2835}$ & $\frac{7984}{155925}$ & $\frac{2672}{1216215}$ & $\frac{%
41462}{638512875}$ & $\frac{15206}{10854718875}$ \\ 
$9$ & $0$ & $0$ & $0$ & $0$ & $0$ & $0$ & $0$ & $0$ & $0$ & $0$%
\end{tabular}%
\end{equation*}

This function is related to the $\cosh t$. That is%
\begin{equation*}
F_{y_{2}}(t,k)=\frac{2}{(2k)!}\left( \cosh t+1\right) ^{k}
\end{equation*}%
By using this function, we get the following combinatorial sums:

\begin{theorem}
\begin{equation*}
y_{2}(n,k;1)=\frac{1}{\left( 2k\right) !}\sum_{j=0}^{k}\left( 
\begin{array}{c}
k \\ 
j%
\end{array}%
\right) 2^{k-j}\sum_{l=0}^{j}\left( 
\begin{array}{c}
j \\ 
l%
\end{array}%
\right) \left( 2l-j\right) ^{2n}
\end{equation*}%
and also%
\begin{equation*}
\sum_{j=0}^{k}\left( 
\begin{array}{c}
k \\ 
j%
\end{array}%
\right) 2^{k-j}\sum_{l=0}^{j}\left( 
\begin{array}{c}
j \\ 
l%
\end{array}%
\right) \left( 2l-j\right) ^{2n+1}=0.
\end{equation*}
\end{theorem}

\begin{proof}
By using (\ref{CC1}), we have%
\begin{equation*}
\sum_{n=0}^{\infty }y_{2}(n,k)\frac{t^{2n}}{\left( 2n\right) !}%
=\sum_{n=0}^{\infty }\left( \frac{1}{\left( 2k\right) !}\sum_{j=0}^{k}\left( 
\begin{array}{c}
k \\ 
j%
\end{array}%
\right) 2^{k-j}\sum_{l=0}^{j}\left( 
\begin{array}{c}
j \\ 
l%
\end{array}%
\right) \left( 2l-j\right) ^{2n}\right) \frac{t^{n}}{n!}.
\end{equation*}%
Comparing the coefficients of $t^{2n}$ on both sides of the above equation,
we arrive at the desired result.
\end{proof}

By using (\ref{CC1}), we obtain%
\begin{equation*}
F_{y_{1}}(t,2k;1)e^{-kt}=\frac{k!}{(2k)!}%
\sum_{v=0}^{k}F_{y_{1}}(t,v;1)F_{y_{1}}(-t,k-v;1).
\end{equation*}%
By using the above functional equation, we obtain the following theorem:

\begin{theorem}
\begin{equation*}
\sum_{j=0}^{n}\left( 
\begin{array}{c}
n \\ 
j%
\end{array}%
\right) (-k)^{n-j}y_{1}(j,2k;1)=\frac{k!}{(2k)!}\sum_{j=0}^{n}(-1)^{n-j}%
\left( 
\begin{array}{c}
n \\ 
j%
\end{array}%
\right) \sum_{v=0}^{k}y_{1}(j,v;1)y_{1}(n-j,k-v;1).
\end{equation*}
\end{theorem}

\begin{lemma}
\label{LemmaRainville}(\cite[Lemma 11, Eq-(7)]{Rainville})%
\begin{equation*}
\sum_{n=0}^{\infty }\sum_{k=0}^{\infty }A(n,k)=\sum_{n=0}^{\infty
}\sum_{k=0}^{\left[ \frac{n}{2}\right] }A(n,n-2k),
\end{equation*}%
where $\left[ x\right] $ denotes the greatest integer function.
\end{lemma}

\begin{theorem}
\label{TheoremG copy(1)}%
\begin{equation*}
y_{1}(n,2k;1)=\sum_{j=0}^{\left[ \frac{n}{2}\right] }\left( 
\begin{array}{c}
n \\ 
2j%
\end{array}%
\right) k^{n-2j}y_{2}(j,k;1).
\end{equation*}
\end{theorem}

\begin{proof}
By using (\ref{CC1}), we obtain the following functional equation:%
\begin{equation*}
F_{y_{1}}(t,2k;1)=F_{y_{2}}(t,k)e^{kt}
\end{equation*}%
Combining this equation with (\ref{ay1}), we get%
\begin{equation*}
\sum_{n=0}^{\infty }y_{1}(n,2k;1)\frac{t^{n}}{n!}=\sum_{n=0}^{\infty }\frac{%
(kt)^{n}}{n!}\sum_{n=0}^{\infty }y_{2}(n,k;1)\frac{t^{2n}}{\left( 2n\right) !%
}.
\end{equation*}%
Applying Lemma \ref{LemmaRainville} in the above equation, we have%
\begin{equation*}
\sum_{n=0}^{\infty }y_{1}(n,2k;1)\frac{t^{n}}{n!}=\sum_{n=0}^{\infty }\left(
\sum_{j=0}^{\left[ \frac{n}{2}\right] }\frac{k^{n-2j}}{(2j)!(n-2j)!}%
y_{2}(j,k;1)\right) t^{n}.
\end{equation*}%
Comparing the coefficients of $t^{n}$ on both sides of the above equation,
we arrive at the desired result.
\end{proof}

We now present an explicit relation between the Lucas numbers $L_{n}$ and
the numbers $y_{2}(n,k;1)$ by the following theorem:

\begin{theorem}
Let $a+b=1$, $ab=-1$ and $\frac{a-b}{2}=c=\frac{\sqrt{5}}{2}$. Then we have%
\begin{equation*}
L_{n}^{(k)}=\sum_{j=0}^{k}\left( 
\begin{array}{c}
k \\ 
j%
\end{array}%
\right) \left( 2j\right) !(-2)^{k-j}\sum_{m=0}^{\left[ \frac{n}{2}\right]
}\left( 
\begin{array}{c}
n \\ 
2m%
\end{array}%
\right) c^{2m}y_{2}(m,j;1)\left( \frac{k}{2}\right) ^{n-2m},
\end{equation*}%
where $L_{n}^{(k)}$ denotes the Lucas numbers of order $k$.
\end{theorem}

\begin{proof}
In \cite[pp. 232-233]{koshy} and \cite{Byrd}, the Lucas numbers $L_{n}$ are
defined by means of the following generating function:%
\begin{equation*}
e^{at}+e^{bt}=\sum_{n=0}^{\infty }L_{n}\frac{t^{n}}{n!}.
\end{equation*}%
From the above, we have%
\begin{equation}
F_{L}(t,k;a,b)=\left( e^{at}+e^{bt}\right) ^{k}=\sum_{n=0}^{\infty
}L_{n}^{(k)}\frac{t^{n}}{n!}.  \label{Lu}
\end{equation}%
where%
\begin{equation*}
L_{n}^{(k)}=\sum_{j=0}^{n}\left( 
\begin{array}{c}
n \\ 
j%
\end{array}%
\right) L_{n}^{(m)}L_{n}^{(k-m)}.
\end{equation*}%
By combining (\ref{Lu}) with (\ref{C1}), we obtain the following functional
equation%
\begin{equation*}
F_{L}(t,k;a,b)=e^{\frac{tk}{2}}\sum_{j=0}^{k}\left( 
\begin{array}{c}
k \\ 
j%
\end{array}%
\right) (-2)^{k-j}(2j)!F_{y_{2}}(ct,j;1).
\end{equation*}%
Since $F_{y_{2}}(ct,j;\lambda )$ is an even function, we have%
\begin{equation*}
\sum_{n=0}^{\infty }L_{n}^{(k)}\frac{t^{n}}{n!}=\sum_{n=0}^{\infty }\left( 
\frac{k}{2}\right) ^{n}\frac{t^{n}}{n!}\sum_{j=0}^{k}\left( 
\begin{array}{c}
k \\ 
j%
\end{array}%
\right) (-2)^{k-j}(2j)!\sum_{m=0}^{\infty }y_{2}(m,j;1)c^{m}\frac{t^{2m}}{%
\left( 2m\right) !}.
\end{equation*}%
Applying Lemma \ref{LemmaRainville} in the above equation, we get%
\begin{equation*}
\sum_{n=0}^{\infty }L_{n}^{(k)}\frac{t^{n}}{n!}=\sum_{n=0}^{\infty
}\sum_{j=0}^{k}\left( 
\begin{array}{c}
k \\ 
j%
\end{array}%
\right) \left( 2j\right) !(-2)^{k-j}\sum_{m=0}^{\left[ \frac{n}{2}\right]
}\left( 
\begin{array}{c}
n \\ 
2m%
\end{array}%
\right) c^{2m}y_{2}(m,j;1)\left( \frac{k}{2}\right) ^{n-2m}\frac{t^{n}}{n!}.
\end{equation*}%
Comparing the coefficients of $t^{n}$ on both sides of the above equation,
we arrive at the desired result.
\end{proof}

We also present an identity including the Fibonacci numbers $f_{n}$, the
Lucas numbers $L_{n}$ and the numbers $y_{1}(n,k;1)$ by the following
theorem:

\begin{theorem}
Let $a+b=1$, $ab=-1$ and $\frac{a-b}{2}=c=\frac{\sqrt{5}}{2}$. Then we have%
\begin{equation*}
L_{n}^{(k)}=k!\sum_{j=0}^{n}\left( 
\begin{array}{c}
n \\ 
j%
\end{array}%
\right) \left( 2c\right) ^{n-j}y_{1}(n-j,k;1)\left( f_{j}\left(
a-2ck^{j}\right) +f_{j-1}\right) .
\end{equation*}
\end{theorem}

\begin{proof}
We set 
\begin{equation*}
F_{f}(t,a,b)=\frac{e^{at}-e^{bt}}{a-b}=\sum_{n=0}^{\infty }f_{n}\frac{t^{n}}{%
n!}
\end{equation*}%
(\textit{cf}. \cite[p. 232]{koshy}, \cite{Byrd}). By combining (\ref{Lu})
and (\ref{ay1}) with the above equation, we obtain the following functional
equation%
\begin{equation*}
F_{L}(t,k;a,b)=k!F_{y_{1}}(2ct,k;1)\left( e^{akt}-2cF_{f}(kt,a,b)\right) .
\end{equation*}%
Therefore 
\begin{eqnarray*}
\sum_{n=0}^{\infty }L_{n}^{(k)}\frac{t^{n}}{n!} &=&k!\sum_{n=0}^{\infty
}\sum_{j=0}^{n}\left( 
\begin{array}{c}
n \\ 
j%
\end{array}%
\right) a^{j}\left( 2c\right) ^{n-j}y_{1}(n-j,k;1)\frac{t^{n}}{n!} \\
&&-k!\sum_{n=0}^{\infty }\sum_{j=0}^{n}\left( 
\begin{array}{c}
n \\ 
j%
\end{array}%
\right) (2c)^{n-j+1}y_{1}(n-j,k;1)k^{j}f_{j}\frac{t^{n}}{n!}.
\end{eqnarray*}%
After some elementary calculations and comparing the coefficients of $\frac{%
t^{n}}{n!}$ on both sides of the above equation, we arrive at the desired
result.
\end{proof}

\subsection{Recurrence relation for the numbers $y_{2}(n,k;\protect\lambda )$%
}

Here, taking derivative of (\ref{C1}), with respect to $t$, we give a
recurrence relation for the numbers $y_{2}(n,k;\lambda )$.

\begin{theorem}
Let $k$ be a positive integer. Then we have%
\begin{equation*}
y_{2}(n+1,k;\lambda )=ky_{2}(n,k;\lambda )-y_{2}(n,k-1;\lambda )-\lambda
^{-1}\sum_{j=0}^{n}\left( 
\begin{array}{c}
n \\ 
j%
\end{array}%
\right) (-1)^{n-j}y_{2}(j,k-1;\lambda ).
\end{equation*}
\end{theorem}

\begin{proof}
Taking derivative of (\ref{C1}), with respect to $t$, we obtain the
following partial differential equation:%
\begin{equation*}
\frac{\partial }{\partial t}F_{y_{2}}(t,k;\lambda )=kF_{y_{2}}(t,k;\lambda
)-F_{y_{2}}(t,k-1;\lambda )-\lambda ^{-1}e^{-t}F_{y_{2}}(t,k-1;\lambda ).
\end{equation*}
Combining (\ref{C1}) with the above equation, we obtain%
\begin{eqnarray*}
\sum_{n=1}^{\infty }y_{2}(n,k;\lambda )\frac{t^{n-1}}{\left( n-1\right) !}
&=&k\sum_{n=0}^{\infty }y_{2}(n,k;\lambda )\frac{t^{n}}{n!}%
-\sum_{n=0}^{\infty }y_{2}(n,k-1;\lambda )\frac{t^{n}}{n!} \\
&&-\sum_{n=0}^{\infty }\sum_{j=0}^{n}\left( 
\begin{array}{c}
n \\ 
j%
\end{array}%
\right) (-1)^{n-j}y_{2}(j,k-1;\lambda )\frac{t^{n}}{n!}.
\end{eqnarray*}%
After some elementary calculation, comparing the coefficients of $\frac{t^{n}%
}{n!}$ on both sides of the above equation, we arrive at the desired result.
\end{proof}

\begin{theorem}
Let $k$ be a positive integer. Then we have%
\begin{equation*}
\frac{\partial }{\partial t}y_{2}(n,k;\lambda )=\lambda y_{2}(n,k;\lambda )-%
\frac{\lambda }{k(2k-1)}y_{2}(n,k-1;\lambda ).
\end{equation*}
\end{theorem}

\begin{proof}
Taking derivative of (\ref{C1}), with respect to $\lambda $, we obtain the
following partial differential equation:%
\begin{equation*}
\frac{\partial }{\partial \lambda }F_{y_{2}}(t,k;\lambda )=\lambda
F_{y_{2}}(t,k;\lambda )-\frac{\lambda }{k(2k-1)}F_{y_{2}}(t,k-1;\lambda ).
\end{equation*}%
Combining (\ref{C1}) with the above equation, we get%
\begin{equation*}
\sum_{n=0}^{\infty }\frac{\partial }{\partial \lambda }y_{2}(n,k;\lambda )%
\frac{t^{n}}{n!}=\sum_{n=0}^{\infty }\lambda y_{1}(n,k;\lambda )\frac{t^{n}}{%
n!}-\frac{\lambda }{k(2k-1)}\sum_{n=0}^{\infty }y_{1}(n,k-1;\lambda )\frac{%
t^{n}}{n!}.
\end{equation*}%
After some elementary calculation, comparing the coefficients of $\frac{t^{n}%
}{n!}$ on both sides of the above equation, we arrive at the desired result.
\end{proof}

\section{$\protect\lambda $-central factorial numbers $C(n,k;\protect\lambda %
)$}

In this section, we define $\lambda $-central factorial numbers $%
C(n,k;\lambda )$ by means of the following generating function:%
\begin{equation}
F_{C}(t,k;\lambda )=\frac{1}{(2k)!}\left( \lambda e^{t}+\lambda
^{-1}e^{-t}-2\right) ^{k}=\sum_{n=0}^{\infty }C(n,k;\lambda )\frac{t^{n}}{n!}
\label{caC0}
\end{equation}%
where $k\in \mathbb{N}_{0}$ and $\lambda \in \mathbb{C}$.

Note that there is one generating function for each value of $k$.

For $\lambda =1$, we have the central factorial numbers%
\begin{equation*}
T(n,k)=C(n,k;1)
\end{equation*}%
(\textit{cf}. \cite{Alayont}, \cite{Cigler}, \cite{Kang}, \cite{AM2014}, 
\cite{SrivastavaLiu}).

\begin{theorem}
\begin{equation*}
T(n,k;\lambda ^{2})=2^{-n}(2k)!\sum_{j=0}^{n}\left( 
\begin{array}{c}
n \\ 
j%
\end{array}%
\right) C(j,k;\lambda )y_{2}(n-j,k;\lambda ).
\end{equation*}
\end{theorem}

\begin{proof}
By using (\ref{C1}) and (\ref{caC0}), we get the following functional
equation:%
\begin{equation*}
F_{C}(2t,k;\lambda ^{2})=(2k)!F_{C}(t,k;\lambda )F_{y_{2}}(t,k;\lambda ).
\end{equation*}%
From this equation, we get%
\begin{equation*}
\sum_{n=0}^{\infty }C(n,k;\lambda ^{2})\frac{\left( 2t\right) ^{n}}{n!}%
=(2k)!\sum_{n=0}^{\infty }C(n,k;\lambda )\frac{t^{n}}{n!}\sum_{n=0}^{\infty
}y_{2}(n,k;\lambda )\frac{t^{n}}{n!}.
\end{equation*}%
Therefore%
\begin{equation*}
\sum_{n=0}^{\infty }C(n,k;\lambda ^{2})\frac{2^{n}t^{n}}{n!}%
=\sum_{n=0}^{\infty }(2k)!\sum_{j=0}^{n}\left( 
\begin{array}{c}
n \\ 
j%
\end{array}%
\right) C(j,k;\lambda )y_{2}(n-j,k;\lambda )\frac{t^{n}}{n!}.
\end{equation*}%
Comparing the coefficients of $\frac{t^{n}}{n!}$ on both sides of the above
equation, we arrive at the desired result.
\end{proof}

By using (\ref{CT-1}) and (\ref{CC1}), we obtain the following functional
equation:%
\begin{equation*}
F_{T}(t,k)F_{y_{2}}(t,k)=\frac{1}{(2k)!}F_{T}(2t,k).
\end{equation*}%
Combining the above equation with (\ref{CT-1}) and (\ref{CC1}), we get%
\begin{equation*}
\frac{1}{(2k)!}\sum_{n=0}^{\infty }T(n,k)\frac{2^{n}t^{n}}{n!}%
=\sum_{n=0}^{\infty }y_{2}(j,k)\frac{t^{n}}{n!}\sum_{n=0}^{\infty }T(n,k)%
\frac{t^{n}}{n!}.
\end{equation*}%
Therefore%
\begin{equation*}
\frac{1}{(2k)!}\sum_{n=0}^{\infty }T(n,k)\frac{2^{n}t^{n}}{n!}%
=\sum_{n=0}^{\infty }\left( \sum_{j=0}^{n}\left( 
\begin{array}{c}
n \\ 
j%
\end{array}%
\right) y_{2}(j,k;1)T(n-j,k)\right) \frac{t^{n}}{n!}
\end{equation*}%
Comparing the coefficients of $\frac{t^{n}}{n!}$ on both sides of the above
equation, we obtain a relationship between the central factorial numbers $%
T(n,k)$ and the numbers $y_{2}(j,k)$ by the following theorem:

\begin{theorem}
\begin{equation*}
T(n,k)=2^{-n}(2k)!\sum_{j=0}^{n}\left( 
\begin{array}{c}
n \\ 
j%
\end{array}%
\right) y_{2}(j,k;1)T(n-j,k).
\end{equation*}
\end{theorem}

\begin{remark}
In \cite{Aloyat-1}, Alayont et al. have studied the rook polynomials, which
count the number of ways of placing non-attacking rooks on a chess board. By
using generalization of these polynomials, they gave the rook number
interpretations of generalized central factorial and the Genocchi numbers.

In \cite{Alayont}, Alayont and Krzywonos gave the following results for the
classical central factorial numbers:

The number of ways to place $k$ rooks on a size $m$ triangle board in three
dimensions is equal to $T(m+1,m+1-k)$, where $0\leq k\leq m$.
\end{remark}

\section{Application in Statistics: In the binomial distribution and the
Bernstein polynomials}

Let $n$ be a nonnegative integer. For every function $f:\left[ 0,1\right]
\rightarrow \mathbb{R}$ and the $n^{th}$ Bernstein polynomial of $f$ is
defined by%
\begin{equation*}
B_{n}(f,x)=\sum_{k=0}^{n}\left( 
\begin{array}{c}
n \\ 
k%
\end{array}%
\right) f\left( \frac{n}{k}\right) B_{k}^{n}(x),
\end{equation*}%
where $B_{k}^{n}(x)$\ denotes the Bernstein basis functions:%
\begin{equation*}
B_{k}^{n}(x)=\left( 
\begin{array}{c}
n \\ 
k%
\end{array}%
\right) x^{k}(1-x)^{n-k}
\end{equation*}%
and $x\in \left[ 0,1\right] $. Let $(U_{k})_{k\geq 1}$ be a sequence of
independent distributed random variable having the uniform distribution on $%
\left[ 0,1\right] $ and defined by Adel \textit{et al}. \cite{Adel JMAA}:%
\begin{equation*}
S_{n}(x)=\sum_{k=1}^{n}1_{\left[ 0,x\right] }(U_{k}).
\end{equation*}%
In \cite{Adel JMAA}, it is well-know that, $S_{n}(x)$ is a binomial random
variable. That is the theory of Probability and Statistics, the binomial
distribution is very useful. This distribution, with parameters $n$ and
probability $x$, is the discrete probability distribution. This distribution
is defined as follows:%
\begin{equation*}
P(S_{n}(x)=k)=\left( 
\begin{array}{c}
n \\ 
k%
\end{array}%
\right) x^{k}(1-x)^{n-k},
\end{equation*}%
where $k=0,1,2,\cdots ,n$. $E$ denotes mathematical expectation, than%
\begin{equation*}
Ef\left( \frac{S_{n}(x)}{n}\right) =B_{n}(f,x)
\end{equation*}%
(\textit{cf}. \cite{Adel JMAA}). For any $x\in \left( 0,1\right) $, $n\geq 2$%
, and $r>1$, Adel \textit{et al}. \cite{Adel JMAA} defined%
\begin{equation}
E\left( S_{n}(x)\right) ^{r}=\sum_{k=0}^{n}\left( 
\begin{array}{c}
n \\ 
k%
\end{array}%
\right) k^{r}x^{k}(1-x)^{n-k}.  \label{ay3}
\end{equation}%
Substituting $x=\frac{1}{2}$ into (\ref{ay3}), we get%
\begin{equation}
E\left( S_{n}\left( \frac{1}{2}\right) \right) ^{r}=\frac{1}{2^{n}}%
\sum_{k=0}^{n}\left( 
\begin{array}{c}
n \\ 
k%
\end{array}%
\right) k^{r}.  \label{ay4}
\end{equation}%
By combining (\ref{ay2}) with (\ref{ay4}), we arrive at the following
theorem:

\begin{theorem}
Let $n\geq 2$. Let $r$ be a positive integer with $r>1$. Then we have%
\begin{equation*}
y_{1}(r,n)=\frac{2^{n}}{n!}E\left( S_{n}\left( \frac{1}{2}\right) \right)
^{r}.
\end{equation*}
\end{theorem}

Integrating (\ref{ay3}) from $0$ to $1$, we get%
\begin{equation*}
\int\limits_{0}^{1}E\left( S_{n}(x)\right) ^{r}dx=\frac{1}{n+1}%
\sum_{k=0}^{n}k^{r}.
\end{equation*}%
By substituting (\ref{BB11}) into the above equation, we arrive at the
following theorem:

\begin{theorem}
\begin{equation*}
\int\limits_{0}^{1}E\left( S_{n}(x)\right) ^{r}dx=\frac{%
B_{r+1}(n+1)-B_{r+1}(0)}{\left( n+1\right) \left( r+1\right) }.
\end{equation*}
\end{theorem}

\section{Computation of the Euler numbers of negative order}

In this section, we not only give elementary properties of the first and the
second kind Euler polynomials and numbers, but also compute the first kind
of Apostol type Euler numbers associated with the numbers $y_{1}(n,k;\lambda
)$ and\ $y_{2}(n,k;\lambda )$.

We define the second kind Apostol type Euler polynomials of order $k$, with $%
k\geq 0$, $E_{n}^{\ast (k)}(x;\lambda )$ by means of the following
generating functions:%
\begin{equation*}
F_{P}(t,x;k,\lambda )=\left( \frac{2}{\lambda e^{t}+\lambda ^{-1}e^{-t}}%
\right) ^{k}e^{tx}=\sum_{n=0}^{\infty }E_{n}^{\ast (k)}(x;\lambda )\frac{%
t^{n}}{n!}.
\end{equation*}%
Substituting $x=0$ into the above equation, we get the second kind Apostol
type Euler numbers of order $k$, $E_{n}^{\ast (k)}(\lambda )$ by means of
the following generating function:%
\begin{equation*}
F_{N}(t;k,\lambda )=\left( \frac{2}{\lambda e^{t}+\lambda ^{-1}e^{-t}}%
\right) ^{k}=\sum_{n=0}^{\infty }E_{n}^{\ast (k)}(\lambda )\frac{t^{n}}{n!}.
\end{equation*}%
If we substitute $k=\lambda =1$ into the above generating function, then we
have%
\begin{equation*}
E_{n}^{\ast }=E_{n}^{\ast (1)}(1).
\end{equation*}

The first kind Apostol-Euler numbers of order $-k$ are defined by means of
the following generating function:%
\begin{equation}
G_{E}(t,-k;\lambda )=\left( \frac{\lambda e^{t}+1}{2}\right)
^{k}=\sum_{n=0}^{\infty }E_{n}^{(-k)}(\lambda )\frac{t^{n}}{n!}.
\label{ae-1}
\end{equation}%
The second kind Apostol type Euler numbers of order $-k$ are defined by
means of the following generating function:%
\begin{equation}
F_{N}(t;-k,\lambda )=\left( \frac{\lambda e^{t}+\lambda ^{-1}e^{-t}}{2}%
\right) ^{k}=\sum_{n=0}^{\infty }E_{n}^{\ast (-k)}(\lambda )\frac{t^{n}}{n!}.
\label{Cac3}
\end{equation}

The numbers $E_{n}^{\ast (-k)}(\lambda )$ are related to the numbers $%
E_{n}^{(-k)}(\lambda )$ and the Apostol Bernoulli numbers $%
B_{n}^{(-k)}(\lambda )$ of the negative order. By using (\ref{Cac3}), we get
the following functional equation:%
\begin{equation*}
F_{N}(t;-k,\lambda )=\sum_{j=0}^{k}\left( 
\begin{array}{c}
k \\ 
j%
\end{array}%
\right) 2^{j-k}t^{k-j}G_{E}(t,-j;\lambda )H_{B}(-t,-k+j;\lambda ^{-1}),
\end{equation*}%
where%
\begin{equation*}
H_{B}(t,-k;\lambda )=\left( \frac{\lambda e^{t}-1}{t}\right)
^{k}=\sum_{n=0}^{\infty }B_{n}^{(-k)}(\lambda )\frac{t^{n}}{n!}\text{ (%
\textit{cf}. \cite{Luo}, \cite{OzdenAMC2014}, \cite{Srivastava2011}).}
\end{equation*}%
By using this equation, we get%
\begin{equation*}
\sum_{n=0}^{\infty }E_{n}^{\ast (-k)}(\lambda )\frac{t^{n}}{n!}%
=\sum_{j=0}^{k}\left( 
\begin{array}{c}
k \\ 
j%
\end{array}%
\right) 2^{j-k}t^{k-j}\sum_{n=0}^{\infty }E_{n}^{(-j)}(\lambda )\frac{t^{n}}{%
n!}\sum_{n=0}^{\infty }B_{n}^{(-k+j)}(\lambda ^{-1})\frac{\left( -t\right)
^{n}}{n!}.
\end{equation*}%
Therefore%
\begin{eqnarray*}
\sum_{n=0}^{\infty }E_{n}^{\ast (-k)}(\lambda )\frac{t^{n}}{n!}
&=&\sum_{n=0}^{\infty }\sum_{j=0}^{k}\left( 
\begin{array}{c}
k \\ 
j%
\end{array}%
\right) \sum_{l=0}^{n-k+j}(-1)^{n+j-k-l}\left( 
\begin{array}{c}
n-k+j \\ 
l%
\end{array}%
\right) \\
&&\times 2^{j-k}(n)_{k-j}E_{l}^{(-j)}(\lambda )B_{n+j-k-l}^{(-k+j)}(\lambda
^{-1})\frac{t^{n}}{n!}.
\end{eqnarray*}%
Comparing the coefficients of $\frac{t^{n}}{n!}$ on both sides of the above
equation, we arrive at the following theorem:

\begin{theorem}
\begin{equation*}
E_{n}^{\ast (-k)}(\lambda )\frac{t^{n}}{n!}=\sum_{j=0}^{k}\left( 
\begin{array}{c}
k \\ 
j%
\end{array}%
\right) \sum_{l=0}^{n-k+j}(-1)^{n+j-k-l}\left( 
\begin{array}{c}
n-k+j \\ 
l%
\end{array}%
\right) 2^{j-k}(n)_{k-j}E_{l}^{(-j)}(\lambda )B_{n+j-k-l}^{(-k+j)}(\lambda
^{-1}).
\end{equation*}
\end{theorem}

After the results of the preceding sections, we are ready to compute the
Euler numbers of negative order.

Combining (\ref{ay1}) and (\ref{ae-1}), we get%
\begin{equation*}
k!2^{k}\sum_{n=0}^{\infty }y_{1}(n,k;\lambda )\frac{t^{n}}{n!}%
=\sum_{n=0}^{\infty }E_{n}^{(-k)}(\lambda )\frac{t^{n}}{n!}.
\end{equation*}%
Comparing the coefficients of $\frac{t^{n}}{n!}$ on both sides of the above
equation, we arrive at the following theorem:

\begin{theorem}
Let $k$ be nonnegative integer. Then we have%
\begin{equation}
E_{n}^{(-k)}(\lambda )=k!2^{-k}y_{1}(n,k;\lambda ).  \label{Cab3}
\end{equation}
\end{theorem}

\begin{remark}
Substituting $\lambda =1$ into (\ref{Cab3}), we obtain the following
explicit formula for the first kind Euler numbers of order $-k$ as follows:%
\begin{equation}
E_{n}^{(-k)}=2^{-k}\sum_{j=0}^{k}\left( 
\begin{array}{c}
k \\ 
j%
\end{array}%
\right) j^{n}  \label{Caa3}
\end{equation}
\end{remark}

For $k=0,-1,-2,\ldots -7$, we compute a few values of the numbers $%
E_{n}^{(-k)}$ given by Equation (\ref{Caa3}) as follows:%
\begin{eqnarray*}
E_{0}^{(0)} &=&1, \\
E_{n}^{(0)} &=&0,\text{(}n\neq 0\text{)} \\
E_{n}^{(-1)} &=&\frac{1}{2}, \\
E_{n}^{(-2)} &=&{2}^{n-2}+\frac{1}{2}, \\
E_{n}^{(-3)} &=&\frac{{3}^{n}}{8}+3.{2}^{n-3}+\frac{3}{8}, \\
E_{n}^{(-4)} &=&\frac{{3}^{n}}{4}+{4}^{n-2}+3.{2}^{n-3}+\frac{1}{4}, \\
E_{n}^{(-5)} &=&\frac{{5}^{n}}{32}+\frac{5.{3}^{n}}{16}+\frac{5.{4}^{n-2}}{2}%
+5.{2}^{n-4}+\frac{5}{32}, \\
E_{n}^{(-6)} &=&\frac{{6}^{n}}{64}+\frac{3.{5}^{n}}{32}+\frac{5.{3}^{n}}{16}%
+15.{4}^{n-3}+15.{2}^{n-6}+\frac{3}{32} \\
E_{n}^{(-7)} &=&\frac{{7}^{n}}{128}+\frac{7.{6}^{n}}{128}+\frac{21.{5}^{n}}{%
128}+\frac{35.{3}^{n}}{128}+\frac{35.{4}^{n-3}}{2}+21.{2}^{n-7}+\frac{7}{128}%
,...
\end{eqnarray*}%
That is for $n=0,1,2,\ldots ,9$ and $k=0,-1,-2,\ldots ,-9$, we compute a few
values of the numbers $E_{n}^{(-k)}$, given by the above relations, as
follows:%
\begin{equation*}
\begin{tabular}{lllllllllll}
$n\backslash k$ & $0$ & $-1$ & $-2$ & $-3$ & $-4$ & $-5$ & $-6$ & $-7$ & $-8$
& $-9$ \\ 
$0$ & $1$ & $\frac{1}{2}$ & $\frac{3}{4}$ & $\frac{7}{8}$ & $\frac{15}{16}$
& $\frac{33}{32}$ & $\frac{33}{64}$ & $\frac{81}{64}$ & $\cdots $ & $\cdots $
\\ 
$1$ & $0$ & $\frac{1}{2}$ & $1$ & $\frac{3}{2}$ & $2$ & $\frac{5}{2}$ & $3$
& $\frac{7}{2}$ & $4$ & $\frac{9}{2}$ \\ 
$2$ & $0$ & $\frac{1}{2}$ & $\frac{3}{2}$ & $3$ & $5$ & $\frac{15}{2}$ & $%
\frac{21}{2}$ & $14$ & $18$ & $\frac{45}{2}$ \\ 
$3$ & $0$ & $\frac{1}{2}$ & $\frac{5}{2}$ & $\frac{27}{4}$ & $14$ & $25$ & $%
\frac{81}{2}$ & $\frac{245}{4}$ & $88$ & $\frac{243}{2}$ \\ 
$4$ & $0$ & $\frac{1}{2}$ & $\frac{9}{2}$ & $\frac{33}{2}$ & $\frac{85}{2}$
& $90$ & $168$ & $287$ & $459$ & $\frac{1395}{2}$ \\ 
$5$ & $0$ & $\frac{1}{2}$ & $\frac{17}{2}$ & $\frac{171}{4}$ & $137$ & $%
\frac{1375}{4}$ & $738$ & $1421$ & $2524$ & $4212$ \\ 
$6$ & $0$ & $\frac{1}{2}$ & $\frac{33}{2}$ & $\frac{231}{2}$ & $\frac{925}{2}
$ & $\frac{5505}{4}$ & $\frac{13587}{4}$ & $7364$ & $14508$ & $26550$ \\ 
$7$ & $0$ & $\frac{1}{2}$ & $\frac{65}{2}$ & $\frac{1287}{4}$ & $1619$ & $%
5725$ & $\frac{65007}{4}$ & $\frac{317275}{8}$ & $86608$ & $173664$ \\ 
$8$ & $0$ & $\frac{1}{2}$ & $\frac{129}{2}$ & $\frac{1833}{2}$ & $\frac{11665%
}{2}$ & $\frac{49155}{2}$ & $\frac{160671}{2}$ & $\frac{441469}{2}$ & $\frac{%
1068453}{2}$ & $1173240$ \\ 
$9$ & $0$ & $\frac{1}{2}$ & $\frac{513}{2}$ & $\frac{15531}{2}$ & $\frac{%
161365}{2}$ & $\frac{1951155}{4}$ & $\frac{8499057}{4}$ & $7418789$ & $%
22071123$ & $\frac{232549335}{4}$%
\end{tabular}%
\end{equation*}

\begin{theorem}
Let $k$ be nonnegative integer. Then we have%
\begin{equation*}
y_{2}(n,k;\lambda )=\frac{2^{k}}{(2k)!}\sum_{l=0}^{k}\left( 
\begin{array}{c}
k \\ 
l%
\end{array}%
\right) E_{n}^{\ast (-l)}(\lambda ).
\end{equation*}
\end{theorem}

\begin{proof}
By using (\ref{C1}) and (\ref{Cac3}), we get the following functional
equation:%
\begin{equation*}
F_{C}(t,k;\lambda )=\frac{1}{(2k)!}\sum_{l=0}^{k}\left( 
\begin{array}{c}
k \\ 
l%
\end{array}%
\right) 2^{k}F_{N}(t;-l,\lambda ).
\end{equation*}%
From this equation, we obtain%
\begin{equation*}
\sum_{n=0}^{\infty }y_{2}(n,k;\lambda )\frac{t^{n}}{n!}=\sum_{n=0}^{\infty
}\left( \frac{1}{(2k)!}\sum_{l=0}^{k}\left( 
\begin{array}{c}
k \\ 
l%
\end{array}%
\right) 2^{k}E_{n}^{\ast (-l)}(\lambda )\right) \frac{t^{n}}{n!}.
\end{equation*}%
Comparing the coefficients of $\frac{t^{n}}{n!}$ on both sides of the above
equation, we arrive at the desired result.
\end{proof}

\begin{theorem}
Let $k$ be nonnegative integer. Then we have%
\begin{equation*}
y_{2}(n,k;\lambda )=\frac{(-1)^{n}2^{n+k}}{(2k)!}\sum_{l=0}^{k}\left( 
\begin{array}{c}
k \\ 
l%
\end{array}%
\right) \lambda ^{-l}E_{n}^{(-l)}\left( -\frac{l}{2};\lambda ^{2}\right) .
\end{equation*}
\end{theorem}

\begin{proof}
By using (\ref{C1}) and (\ref{Cad3}), we get the following functional
equation:%
\begin{equation*}
F_{C}(t,k;\lambda )=\frac{2^{k}}{(2k)!}\sum_{l=0}^{k}\left( 
\begin{array}{c}
k \\ 
l%
\end{array}%
\right) \lambda ^{-l}F_{P1}\left( 2t,-\frac{l}{2};-k,\lambda ^{2}\right) .
\end{equation*}%
From this equation, we obtain%
\begin{equation*}
\sum_{n=0}^{\infty }y_{2}(n,k;\lambda )\frac{t^{n}}{n!}=\sum_{n=0}^{\infty
}\left( \frac{2^{n+k}}{(2k)!}\sum_{l=0}^{k}\left( 
\begin{array}{c}
k \\ 
l%
\end{array}%
\right) E_{n}^{(-l)}\left( -\frac{l}{2};\lambda ^{2}\right) \right) \frac{%
t^{n}}{n!}.
\end{equation*}%
Comparing the coefficients of $\frac{t^{n}}{n!}$ on both sides of the above
equation, we arrive at the desired result.
\end{proof}

By applying derivative operator to the generating function in (\ref{ay1}),
we give a relationship between the numbers $y_{1}(n,k;\lambda )$\ and $%
E_{n}^{(-1)}\left( \lambda \right) $\ by the following theorem:

\begin{theorem}
Let $k\geq 2$. Then we have%
\begin{eqnarray*}
y_{1}(n+2,k;\lambda ) &=&ky_{1}(n,k;\lambda )+y_{1}(n,k-2;\lambda
)-y_{1}(n,k-1;\lambda ) \\
&&+2k\sum_{l=0}^{n}\left( 
\begin{array}{c}
n \\ 
l%
\end{array}%
\right) E_{l}^{(-1)}\left( \lambda \right) y_{1}(n-l,k;\lambda ).
\end{eqnarray*}
\end{theorem}

\begin{proof}
By applying derivative operator to (\ref{ay1}) with respect to $t$, we
obtain the following partial differential equation:%
\begin{equation*}
\frac{\partial ^{2}}{\partial t^{2}}F_{y_{1}}(t,k;\lambda )=k\left(
F_{y_{1}}(t,k;\lambda )+2G_{E}(t,1;\lambda )\right) +F_{y_{1}}(t,k-2;\lambda
)-F_{y_{1}}(t,k-1;\lambda ).
\end{equation*}

Combining (\ref{ay1}) and (\ref{ae-1}) with the above equation, we get%
\begin{eqnarray*}
&&\sum_{n=2}^{\infty }y_{1}(n,k;\lambda )\frac{t^{n-2}}{\left( n-2\right) !}
\\
&=&k\sum_{n=0}^{\infty }y_{1}(n,k;\lambda )\frac{t^{n}}{n!}%
+\sum_{n=0}^{\infty }y_{1}(n,k-2;\lambda )\frac{t^{n}}{n!}%
-\sum_{n=0}^{\infty }y_{1}(n,k-1;\lambda )\frac{t^{n}}{n!} \\
&&+2k\sum_{n=0}^{\infty }\sum_{l=0}^{n}\left( 
\begin{array}{c}
n \\ 
l%
\end{array}%
\right) E_{l}^{(-1)}\left( \lambda \right) y_{1}(n-l,k;\lambda )\frac{t^{n}}{%
n!}.
\end{eqnarray*}%
We make arrangement of the series and then compare the coefficients of $%
\frac{t^{n}}{n!}$ on both sides of the above equation, and we obtain the
desired result.
\end{proof}

\section{Algorithms and Computation}

Computer science and Applied Mathematics have study information and
computation and also their theoretical foundations. In these areas practical
techniques are very important (\textit{cf}. \cite{Wpc}). Therefore
algorithmic processes play a very important role in both areas. Thus, in
this section, we give two algorithms for the computation of the numbers $%
y_{1}(n,k;\lambda )$ and $y_{2}(n,k;\lambda )$.

\section{Combinatorial applications and further remarks}

In this section, we discuss some combinatorial interpretations of these
numbers, as well as the generalization of the central factorial numbers
given in Section 3-5. These interpretations includes the rook numbers and
polynomials and combinatorial interpretation for the numbers $y_{1}(n,k)$.
We see that our numbers are associated with known counting problems. By
using counting techniques and generating functions techniques, Bona \cite%
{Bona} rederived several known properties and novel relations involving
enumerative combinatorics and related areas. A very interesting further
special cases of the numbers $B(n,k)$ is worthy of note by the work of Bona 
\cite{Bona}. That is, in \cite[P. 46, Exercise 3-4]{Bona}, Bona gave the
following two exercises which are associated with the numbers $B(n,k)$:

\textbf{Exercise 3}. Find the number of ways to place $n$ rooks on an $%
n\times n$ chess board so that no two of them attack each other.

\textbf{Exercise 4}. How many ways are there to place some rooks on an $%
n\times n$ chess board so that no two of them attack each other?

\begin{remark}
Our numbers occur in combinatorics applications. In \cite[P. 46, Exercise 3-4%
]{Bona}, Bona gave detailed and descriptive solution of these two exercises,
which are related to the numbers $B(n,k)$, respectively as follows:

There has to be one rook in each column. The first rook can be anywhere in
its column ($n$ possibilities). The second rook can be anywhere in its
column except in the same row where the first rook is, which leaves $n-1$
possibilities. The third rook can be anywhere in its column, except in the
rows taken by the first and second rook, which leaves $n-2$ possibilities,
and so on, leading to $n(n-1)\cdots 2.1=n!$ possibilities.

Exercise 4. If we place $k$ rooks, then we first need to choose the $k$
columns in which these rooks will be placed. We can do that in $\left( 
\begin{array}{c}
n \\ 
k%
\end{array}%
\right) $ ways. Continuing the line of thought of the solution of the
previous exercise, we can then place our $k$ rooks into the chosen columns
in $(n)_{k}$ ways. Therefore, the total number of possibilities is%
\begin{equation*}
\sum_{k=1}^{n}\left( 
\begin{array}{c}
n \\ 
k%
\end{array}%
\right) (n)_{k}.
\end{equation*}
\end{remark}

\begin{remark}
In (\ref{Ci8}), for $j<d$, it is well-known that%
\begin{equation*}
\left( 
\begin{array}{c}
j \\ 
d%
\end{array}%
\right) =0.
\end{equation*}%
Therefore, we arrive at solutions of Exercise 16 (a) in \cite[p. 55,
Exercise 16(a)]{Bona} and also Exercise 10 \cite[p. 126]{Charamb} as follows:%
\begin{equation*}
2^{n-k}\left( 
\begin{array}{c}
k \\ 
d%
\end{array}%
\right) =\sum_{j=d}^{k}\left( 
\begin{array}{c}
k \\ 
j%
\end{array}%
\right) \left( 
\begin{array}{c}
j \\ 
d%
\end{array}%
\right) .
\end{equation*}
\end{remark}

\section{Conclusions}

In this paper, we have constructed some new families of special numbers with
their generating functions. We give many properties of these numbers. These
numbers are related to the many well-known numbers, which are the Bernoulli
numbers, the Stirling numbers of the second kind, the central factorial
numbers and also related to the Golombek's problem \cite{golombek}
\textquotedblleft Aufgabe 1088\textquotedblright . We have discussed some
combinatorial interpretations of these numbers. Besides, we give some
applications about not only rook polynomials and numbers, but also
combinatorial sum.

\begin{acknowledgement}
The paper was supported by the \textit{Scientific Research Project
Administration of Akdeniz University.}
\end{acknowledgement}


\begin{thebibliography}{99}
\bibitem{Adel JMAA} J. A. Adell, J. Bustamante and J. M. Quesada, Estimates
for the moments of Bernstein polynomials, J. Math. Anal. Appl. 432 (2015)
114-128.

\bibitem{Alayont} F. Alayont and N. Krzywonos, Rook polynomials in three and
higher dimensions, preprint.

\bibitem{Aloyat-1} F. Alayont, R. Moger-Reischer and R. Swift, Rook number
interpretations of generalized central factorial and Genocchi numbers,
preprint.

\bibitem{Bona} M. Bona, Introduction to Enumerative Combinatorics, \textit{%
The McGraw-Hill Companies, Inc.} New York, 2007

\bibitem{Boyadzhiev} K. N. Boyadzhiev, Binomial transform and the backward
difference, http://arxiv.org/abs/1410.3014v2.

\bibitem{Bayad} A. Bayad, Y. Simsek and H. M. Srivastava, Some array type
polynomials associated with special numbers and polynomials, Appl. Math.
Comput. 244 (2014), 149-157.

\bibitem{Byrd} P.F. Byrd, New Relations between Fibonacci and Bernoulli
Numbers, Fibonacci Quarterly 13(1975), 111-114.

\bibitem{Cakic} N. P. Cakic and G. V. Milovanovic, On generalized Stirling
numbers and polynomials, Mathematica Balkanica 18 (2004), 241-248.

\bibitem{Chan} C.-H. Chang and C.-W. Ha, A multiplication theorem for the
Lerch zeta function and explicit representations of the Bernoulli and Euler
polynomials, J. Math. Anal. Appl. 315 (2006), 758-767.

\bibitem{Charamb} C. A. Charalambides, Ennumerative Combinatorics, \textit{%
Chapman\&Hall/Crc, Press Company}, London, New York, 2002.

\bibitem{Cigler} J. Cigler, Fibonacci polynomials and central factorial
numbers, preprint.

\bibitem{Comtet} L. Comtet, Advanced Combinatorics: The Art of Finite and
Infinite Expansions, Reidel, Dordrecht and Boston, 1974 (\textit{Translated
from the French by J. W. Nienhuys}).

\bibitem{Doubilet} P. Doubilet, G.-C. Rota, and R. Stanley, On the
foundations of combinatorial theory. VI. The idea of generating function,
Berkeley Symp. on Math. Statist. and Prob. Proc. Sixth Berkeley Symp. on
Math. Statist. and Prob.(Univ. of Calif. Press 2 (1972), 267-318.

\bibitem{golombek} R. Golombek, Aufgabe 1088, El. Math. 49 (1994) 126-127.

\bibitem{Grademir} G. B. Djordjevic and G. V. Milovanovic, Special classes
of polynomials, University of Nis, Faculty of Technology Leskovac, 2014.

\bibitem{Kang} J. Kang and C. Ryoo, A research on the new polynomials
involved with the central factorial numbers, Stirling numbers and others
polynomials, J. Inequalities Appl. 2014, 26.

\bibitem{Knuth} Donald E. Knuth, The Art of Computer Programming, Volume 1
Fundamental Algorithms (Third Edition) \textit{Addison-Wesley}. ISBN
0-201-89683-4.

\bibitem{T. Kim} T. Kim, $q$-Volkenborn integration, Russian J. Math. Phys.
19 (2002), 288-299.

\bibitem{DSkim1} D. S. Kim, T. Kim, Daehee numbers and polynomials, Appl.
Math. Sciences 7, (2013), 5969-5976.

\bibitem{koshy} T. Koshy, Fibonacci and Lucas numbers with applications,
John Wiley \& Sons, New York, 2001.

\bibitem{Luo} Q. M. Luo, and H. M. Srivastava, Some generalizations of the
Apostol-Genocchi polynomials and the Stirling numbers of the second kind.
Appl. Math. Comput. 217 (2011), 5702-5728.

\bibitem{OzdenAMC2014} H. Ozden and Y. Simsek, Modification and unification
of the Apostol-type numbers and polynomials and their applications, Appl.
Math. Comput. 235 (2014), 338-351.

\bibitem{Ozden} H. Ozden, Y. Simsek, and H. M. Srivastava, A unified
presentation of the generating functions of the generalized Bernoulli, Euler
and Genocchi polynomials. Comput. Math. Appl. 60 (2010), 2779-2787.

\bibitem{Ozden AML} H. Ozden and Y. Simsek, A new extension of $q$-Euler
numbers and polynomials related to their interpolation functions, Appl.
Math. Lett. 21(9) (2008), 934-939.

\bibitem{Rainville} E. D. Rainville, Special functions, \textit{The
Macmillan Company}, New York, 1960.

\bibitem{Ryoo} C. S. Ryoo, T. Kim, and L.-C. Jang, Some Relationships
between the analogs of Euler numbers and polynomials, J. Inequalities and
Appl. 2007(86052) (2007), 1-22.

\bibitem{YsimsekKim} Y. Simsek, On $q$-deformed Stirling numbers, Int. J.
Math. Comput. 15, (2012), 70-80.

\bibitem{RJMP2010} Y. Simsek, Special functions related to Dedekind-type
DC-sums and their applications, Russ. J. Math. Phys. 17 (4) (2010), 495-508.

\bibitem{SimsekMANISA} Y. Simsek, Identities associated with generalized
Stirling type numbers and Eulerian type polynomials, Mathematical Comput.
Appl. 18 (2013), 251-263.

\bibitem{SimsekFPTA} Y. Simsek, Generating functions for generalized
Stirling type numbers, array type polynomials, Eulerian type polynomials and
their alications, Fixed Point Theory Al. 87 (2013), 343-1355.

\bibitem{AM2014} Y. Simsek, Special numbers on analytic functions, Alied
Mathematics 5 (2014), 1091-1098.

\bibitem{mmas2015} Y. Simsek, Analysis of the Bernstein basis functions: an
approach to combinatorial sums involving binomial coefficients and Catalan
numbers, Math. Meth. Appl. Sci.38 (2015), 3007-3021.

\bibitem{Spevy} M. Z. Spivey, Combinatorial Sums and Finite Differences,
Discrete Math. 307(24) (2007), 3130-3146.

\bibitem{M. Spiegel} M. R. Spiegel, Calculus of finite differences and
difference equations, \textit{Schaum's Outline Series in Mathematics,
McGraw-Hill Book Company}, London, Toronto, 1971.

\bibitem{Srivastava2011} H. M. Srivastava, Some generalizations and basic
(or $q$-) extensions of the Bernoulli, Euler and Genocchi polynomials, Appl.
Math. Inform. Sci. 5 (2011), 390-444.

\bibitem{SrivatavaCem} H. M. Srivastava, M. A. Ozarslan and C. Kaano\u{g}lu,
Some families of generating functions for a certain class of three-variable
polynomials, Integral Transforms Spec. Funct. 21 (2010), 885-896.

\bibitem{SrivastavaChoi2012} H. M. Srivastava and J. Choi, Zeta and $q$-Zeta
Functions and Associated Series and Integrals, \textit{Elsevier Science
Publishers}, Amsterdam, London and New York, 2012.

\bibitem{srivas18} H. M. Srivastava, T. Kim. and Y. Simsek, $q$-Bernoulli
numbers and polynomials associated with multiple $q$-zeta functions and
basic $L$-series, Russian J. Math. Phys. 12 (2005), 241-268.

\bibitem{SrivastavaLiu} H. M. Srivastava and G.-D. Liu, Some identities and
congruences involving a certain family of numbers, Russian J. Math. Phys. 16
(2009), 536-542.

\bibitem{SrivastavaBook} H. M. Srivastava and H. L. Manocha, A Treatise on
Generating Functions, \textit{Ellis Horwood Limited Publisher}, Chichester
1984.

\bibitem{Qi} C.-F. Wei and F. Qi, Several closed expressions for the Euler
Numbers, J. Inequalities and Appl. 2015 (219) (2015), 1-8.

\bibitem{Wm} https://en.wikipedia.org/wiki/Generating\_function

\bibitem{Wp} https://en.wikipedia.org/wiki/Generating\_function\_(physics)

\bibitem{Wpc} https://en.wikipedia.org/w/index.php?title=Computer\_science%
\&oldid=320473980
\end{thebibliography}
\end{document}